\documentclass[12pt]{amsart}

\usepackage{amsfonts}
\usepackage{amsmath}
\usepackage{amssymb}
\usepackage[latin1]{inputenc}
\usepackage{a4wide}
\usepackage{graphicx}
\usepackage{amsthm}

\numberwithin{equation}{section}

\newtheorem{thm}{Theorem}[section]
\newtheorem{lem}[thm]{Lemma}
\newtheorem{definition}[thm]{Definition}

\newtheorem{prop}[thm]{Proposition}
\newtheorem{rem}[thm]{Remark}
\newtheorem{Hypothese}{Assumption}

\newenvironment{hyp}{\begin{Hypothese}}{\end{Hypothese}}

\theoremstyle{definition}

\newcommand{\cA}{\ensuremath{\mathcal A}}

\newcommand{\cC}{\ensuremath{\mathcal C}}

\newcommand{\cF}{\ensuremath{\mathcal F}}

\newcommand{\cK}{\ensuremath{\mathcal K}}

\newcommand{\cM}{\ensuremath{\mathcal M}}

\newcommand{\cQ}{\ensuremath{\mathcal Q}}

\newcommand{\cS}{\ensuremath{\mathcal S}}

\newcommand{\frm}{\ensuremath{\mathfrak m}}

\newcommand{\bbH}{{\ensuremath{\mathbb H}} }

\newcommand{\bbL}{{\ensuremath{\mathbb L}} }

\newcommand{\bbR}{{\ensuremath{\mathbb R}} }

\newcommand{\ga}{\alpha}
\newcommand{\gb}{\beta}
\newcommand{\gga}{\gamma}            
\newcommand{\gd}{\delta}
\newcommand{\gep}{\varepsilon}       
\newcommand{\gp}{\varphi}
\newcommand{\gr}{\rho}

\newcommand{\gD}{\Delta}
\newcommand{\gk}{\kappa}

\newcommand{\gs}{\sigma}

\newcommand{\p}{\mathbb{P}}

\newcommand{\Wtn}{{Q_t^N}}

\newcommand{\etats}{\{0,1\}^{2N+1}}

\newcommand{\rond}[1]{\mathcal{#1}}

\newcommand{\J}{J_{G,H,\epsilon}}

\newcommand{\bra}{\langle}
\newcommand{\ket}{\rangle}

\newcommand{\dQ}{{\partial_t Q}}
\newcommand{\dK}{{\partial_t K}}
\newcommand{\reg}{{\text{reg}}}

\newcommand{\ggab}{\gb}

\title[Large deviations of the empirical currents]{Large deviations of the empirical currents for a boundary driven reaction diffusion model}
\date{\today}
\author{Thierry BODINEAU}
\address{Ecole Normale Sup\'erieure, DMA, 45 rue d'Ulm
75230 Paris cedex 05, France}
\author{Maxime LAGOUGE}
\thanks{We are deeply indebted to L. Bertini, B. Derrida, D. Hilhorst, C. landim, J. Lebowitz for many enlightening discussions and useful suggestions. TB acknowledges the support of the French Ministry of Education through the ANR BLAN07-2184264 grant.
This work was partially supported by the NSF Grant DMR-044-2066 and AFOSR Grant AF-FA9550-04 during a stay at Rutgers University and by  the {\it Florence Gould Foundation Endowment} during a stay at the Institute for Advanced Study.}


\begin{document}

\maketitle

\begin{abstract}
We derive a large deviation principle for the empirical currents of lattice gas dynamics which combine a fast stirring mechanism (Symmetric Simple Exclusion Process) and  creation/annihilation mechanisms (Glauber dynamics). Previous results on the density large deviations can be recovered from this general large deviation principle. The contribution of external driving forces due to reservoirs at the boundary of the system is also taken into account.
\end{abstract}

\section{Introduction}

A major challenge in non-equilibrium statistical physics is to
provide a counterpart to the notion of free energy and to the Gibbs
measure which is the cornerstone of the equilibrium statistical
physics. Large deviation principles have been proposed as a good
alternative to investigate properties of non equilibrium systems
\cite{gallavotti,derrida,BDGJL3}. In particular, a lot of attention
has been devoted to the case of lattice gas dynamics for which
explicit large deviation principles can be derived in the
hydrodynamic scaling (we refer the reader to  \cite{derrida,BDGJL3}
for recent surveys and further references). Motivated by these
recent progress in physics, the original mathematical works
\cite{KOV,DV}  on the hydrodynamic large deviations for conservative
dynamics have been generalized to take into account the contribution
of reservoirs at the boundary of the system \cite{BDGJL1, BLM,FLM} and
the current of particles flowing in the system \cite{BDGJL2}.

\medskip

An interesting class of models has been introduced in \cite{DFL} to
describe reaction  diffusion equations by combining the Symmetric
Simple Exclusion Process (SSEP) to a Glauber dynamics which models
the annihilation and creation of particles. In \cite{DFL}, the
hydrodynamic  limit as well as the fluctuations of the density have
been investigated for these models. The density hydrodynamic large
deviations have then been proved in \cite{JLV}. In this paper, we
generalize this result by deriving the joint large deviations of the
density and of the (conservative and non-conservative) currents
flowing in the system. We also take into account the contribution of
reservoirs acting at the boundary of the system. Our results were
motivated by the recent research in non-equilibrium statistical physics on
dissipative dynamics and in particular on granular media
\cite{bertin,levine1,levine2}. We refer to \cite{BL} for a more
comprehensive discussion on the physical aspects of the large
deviations for dissipative systems.

\medskip

Contrary to the purely conservative dynamics \cite{BDGJL2}, one has
to introduce two types of currents: the conservative integrated
current $Q_t$ which records the particle jumps from the diffusive
part of the dynamics (SSEP) and the non-conservative integrated
current $K_t$ associated to the creation annihilation process
(Glauber). Heuristically, if one denotes by  $\dot Q_t(r)$ and $\dot
K_t (r)$ the instantaneous currents at time $t$ and location $r$,
then the density obeys the following equation
\begin{eqnarray*}
\partial_t \rho_t(r) = - \partial_r \dot Q_t(r) + \dot K_t (r) \, .
\end{eqnarray*}
The goal of this paper is to compute the asymptotic cost of
observing an atypical trajectory of the currents $Q,K$ and of the
density $\rho$ when the number of particles tends to infinity. 
Even so it is apparently more complicated to consider the joint
deviations of the empirical currents and the empirical density, 
it turns out that the structure of the joint large deviation functional $I_0$ is more transparent as
it splits into two distinct contributions involving either the
diffusive part or the Glauber part of the dynamics
\begin{eqnarray*}
I_0 (\rho, Q, K) = I_1 (\rho, Q) + I_2 (\rho, K) \,.
\end{eqnarray*}
The precise form of the functional $I_0$  can be found in
\eqref{I0}. In section \ref{sec: densiteGD}, the density large deviation
functional  derived in \cite{JLV} is recovered by
a contraction principle. 
This provides a natural interpretation of  the density large deviation functional
 as the optimal combination between the two macroscopic currents $Q, K$ in order to create
 the atypical density trajectory $\rho$ at a minimal cost $I_1 + I_2$.

\medskip

Our proof relies on the standard machinery developed to study hydrodynamic large deviations \cite{KOV, KL}, as well as on more 
recent tools introduced in  \cite{BDGJL1,BDGJL2,BLM,FLM}. Therefore in this paper, we will not detail the aspects of the proof which can be deduced readily from the existing literature and we will focus on the new features occurring from the non-conservative part of the dynamics.
The paper is organized as follows. In section \ref{sec: Notations and Results},
we introduce the model and state the main results. A strong form of local equilibrium is 
stated in section \ref{sec: A super-exponential replacement lemma}. 
The upper and lower bound of the joint  density/current large deviations are derived 
in sections \ref{sec: Large deviations upper bound} and \ref{sec: Large deviations lower bound}.
Finally the density large deviations are recovered in section \ref{sec: densiteGD}.

\section{Notations and Results}
\label{sec: Notations and Results}

\subsection{The microscopic dynamics and Notations}
\label{sec: microscopic}

We consider the one-dimensional Symmetric Simple Exclusion Process (SSEP) in the domain $\{-N,\dots,N\}$ with creation and annihilation of particles in the bulk and reservoirs at the boundaries. 
More precisely, the particles perform random walks with an exclusion constraint which imposes at most one particle per site and particles can be removed or created in the bulk according to a rate which depends on the local configurations. At the boundaries $\pm N$, two reservoirs maintain constant densities.
The dynamics can be viewed as a toy model for chemical reactions where the chemicals are injected at the boundaries, then diffuse and react in the system \cite{DFL}.

The stochastic dynamics is a Markov process on $\etats$ whose  generator is obtained by adding the generators of the different dynamics
\begin{eqnarray}
L_N = \frac{N^2}{2}L_{0,N}+\frac{N^2}{2}L_{+,N}+\frac{N^2}{2}L_{-,N}+L_{1,N} 
\, ,
\label{generateur}
\end{eqnarray}
with the SSEP generator
\begin{eqnarray*}
L_{0,N} f(\eta) = \sum_{x=-N}^{N-1}\left[f(\eta^{x, x+1})-f(\eta)\right], 
\end{eqnarray*}
and creation and annihilation generators at the boundaries depending on the parameters $\ggab^+$ and $\ggab^-$
\begin{align*}
\begin{split}
L_{+,N}f(\eta) &= \left[\eta(N)+\ggab_+(1-\eta(N)\right]\left[f(\eta^N)-f(\eta)\right]\\
L_{-,N}f(\eta) &= \left[\eta(-N)+\ggab_-(1-\eta(-N)\right]\left[f(\eta^{-N})-f(\eta)\right]
\end{split}
\end{align*}
where 
$$
\eta^x(z)=\left\{\begin{array}{lc}\eta(z) &\textrm{ if } z\neq x\\
1-\eta(z) &\textrm{ if }z=x
\end{array}\right.
;\qquad \eta^{x,y}(z)=\left\{\begin{array}{lc}\eta(y) &\textrm{ if } z=x\\
\eta(x) &\textrm{ if }z=y\\
\eta(z)&\textrm{ else}
\end{array}\right. .
$$ 
Finally the creation and annihilation generator in the bulk is given by
\begin{eqnarray*}
L_{1,N}f(\eta) = \sum_{x=-N+M+1}^{N-M-1} c(x,\eta)  \left[ f(\eta^x)- f(\eta) \right] \, ,
\end{eqnarray*}
where the rate of creation and annihilation $c(x,\cdot)$ is a non negative cylindric function with range $M$, i.e. 
there exists a fixed integer $M$ (possibly equal to 0) such that
$c(x,\eta) = c(\eta_{x-M}, \dots, \eta_{x+M})$ depends only of the values of $\eta$ in $\{ x-M, \dots, x+M \}$. 
Remark that, the diffusive part of the process is speeded up by $N^2$ to obtain a non trivial hydrodynamic evolution.

\medskip

For a given trajectory $\eta\, : \,  [0,T]\to\etats$, let
$\rho_t^N$ be the empirical density of particles in $[-1,1]$ at time $t \in [0,T]$ 
$$
\rho_t^N  =\frac{1}{N} \sum_{x=-N}^N\eta_t(x) \delta_{x/N}.
$$
We denote by $Q_{t}^N(x)$ the conservative current through the edge
$(x,x+1)$, i.e. the total number of particles that have jumped from
$x$ to $x+1$ minus the total number of particles that have jumped
from $x+1$ to $x$ between the times $0$ and $t$.
The empirical measure associated to this current is defined as the signed measure on $[-1,1]$ 
\begin{eqnarray}
\label{eq: qtn}
\Wtn=\dfrac{1}{N ^2}\sum_{x=-N}^{N-1}Q_t^N(x)\delta_{x/N}.
\end{eqnarray}
The renormalization by $N^2$ takes into account the space renormalization as well as the diffusive scaling of the SSEP dynamics which leads to an extra factor $N$. 
We denote by $K_t^N(x)$ the non-conservative current at site $x$, i.e. the total number
of particles created minus the total number of particles annihilated at site $x$ between times
 $0$ and $t$. The corresponding empirical measure is
$$
K^N_t =\dfrac{1}{N}  \sum_{x=-N}^{N}K_t^N(x)\delta_{x/N}.
$$

For any continuous function $\varphi\in C([-1,1])$, we will use the notation 
$$
\bra \rho^N_t \, {\varphi} \ket
= \frac{1}{N}\sum_{x=-N}^{N-1} \eta_{t} (x)\varphi\left(\frac{x}{N}\right).
$$
The same notation will be used for $K^N_t$.
As the conservative current applies to edges, we will write for $\varphi \in C^1([-1,1])$
$$
\bra \Wtn \, \nabla \varphi \ket
= \frac{1}{N}\sum_{x=-N}^{N-1} Q_{t}^N(x) 
\left(  \varphi\left(\frac{x+1}{N}\right)  -\varphi\left(\frac{x}{N}\right) \right).
$$
Note that $\varphi\left(\frac{x+1}{N}\right)  -\varphi\left(\frac{x}{N}\right)$ is of order $1/N$ so that the scaling is coherent with \eqref{eq: qtn}.
Finally, for any functions $f(s,r)$ in $[0,T] \times [-1,1]$, we use the shorthand notation
\begin{eqnarray*}
\forall s \in [0,T], \qquad  
\bra f_s \ket = \int_{-1}^1 \, dr \, f(s,r) \, .
\end{eqnarray*}

\medskip

The density profiles bounded away from 0 and 1 will be relevant so that  
we introduce $C_e([-1,1])$ the set of continuous functions $f$ on $[-1,1]$ for which
 there exists a constant $\epsilon>0$ such that $\epsilon < f < 1-\epsilon$.
Given a function $\gamma \in C_e([-1,1])$,
 let $\nu_\gamma^N$ be the Bernoulli product measure on $\{-N,N\}$ with marginals 
\begin{equation} 
\forall k\in\{-N,N\}, \qquad \nu_\gamma^N(\eta(k)=1)=\gamma( \frac{k}{N}).
\label{mesureproduit}
\end{equation}
For  $\alpha\in [0,1]$, the Bernoulli product measure with uniform  density $\ga$ is denoted by $\nu^N_\alpha$.
Let $\p^N_\gamma$ be the probability measure 
associated to the Markov process $(\rho_t^N,Q_t^N,K_t^N)$ on $[0,T]$
with initial measure $\nu_\gamma^N$ on the particle configurations.

\medskip

We define $\cM$ the set of signed measures on $[-1,1]$ endowed with the weak topology. 
We also consider $\cM_0$ the subset of $\cM$ of all absolutely continuous measures
wrt the Lebesgue measure with positive density bounded by 1
\begin{equation*} 
\cM_0 = \{ \gr(x) dx \in \cM, \qquad 0 \leq \gr(x) \leq 1 \quad a.e. \} \, .
\end{equation*}
In order to consider the joint large deviations of  $(\rho_t^N,Q_t^N,K_t^N)$ during the time interval $[0,T]$, we will work on  $\mathcal{E}=D \big([0,T], \cM_0 \times \cM \times \cM \big)$ the space of cad-lag trajectories with values in 
$\cM_0 \times \cM \times \cM$ endowed with the Skorohod topology \cite{EK}.

\subsection{The results}

The hydrodynamic behavior of the microscopic dynamics introduced in section \ref{sec: microscopic}, can be described in terms of a few macroscopic parameters
\begin{align}
\forall \ga \in [0,1], \qquad 
C(\ga)=\nu_\ga \big(c(0,\eta)(1-\eta(0)\big), \qquad 
A(\ga)=\nu_\ga \big(c(0,\eta)\eta(0)\big) \, ,
\label{CetA}
\end{align}
where $C$ and $A$ represent the average creation and 
annihilation rates at density $\ga$ and 
\begin{eqnarray}
\label{eq: conductivity}
\forall \ga \in [0,1], \qquad  \gs(\ga)= \nu_\ga (\eta_0 (1-\eta_1) ) = \ga (1- \ga) \, ,
\end{eqnarray}
which is the conductivity of the SSEP.
Finally, we denote by $\bar \rho_{\pm} =\dfrac{\ggab_{\pm}}{1+\ggab_{\pm}}$ the densities at the boundaries imposed by the reservoirs.

\medskip

We are now ready to state the hydrodynamic limit 

\begin{thm}
\label{thm: limitehydro}
Let $\gamma\in C_e([-1,1])$ with $\gamma(-1)= \bar \gr_-$ and $\gamma(1)= \bar \gr_+$. 
For each $T>0$ and $\varphi\in C([0,T] \times [-1,1])$,
$$\forall  \delta>0, \qquad
\lim_{N\to\infty} \p^N_\gamma\left[
\left| \int_0^T dt  \bra \rho_t^N\ \varphi_t \ket - \bra \bar \rho_t \; \varphi_t \ket \right| > \delta \right]=0,
$$
where $\bar \rho(t,x)$ is the unique weak solution of 
\begin{equation}
\left\{\begin{array}{lcl}
\partial_t  \bar \rho (t,x) &=& \frac{1}{2}\Delta \bar  \rho (t,x) +C(\bar  \rho(t,x))-A(\bar  \rho(t,x))\\
\bar  \rho(t,\pm 1)&=& \bar  \rho_\pm\\
\bar  \rho(0,x)&=&\gamma(x)
\end{array}\right.  
\label{limitehydro}
\end{equation}
\end{thm}
The meaning of weak solution of (\ref{limitehydro}) is recalled in the Appendix (with $H= G = 0$).

\medskip

As a consequence of Theorem \ref{thm: limitehydro} a law of large numbers holds for the currents.
\begin{thm}
For any test function $\varphi\in C^1([-1,1])$ and $t >0$
\begin{eqnarray*}
\forall \gd >0,
\qquad  && 
\lim_{N\to\infty} \p^N_\gamma
\left[ \left| \bra Q_t^N \, \varphi \ket + \dfrac{1}{2} \int_0^t \, ds \, \bra \varphi \, \nabla \bar  \rho_s \ket \right| > \delta \right] =0 \, ,\\
&&
\lim_{N\to\infty} \p^N_\gamma\left[
\left| \bra K_t^N \, \varphi \ket -\int_0^t \, ds \, \big\bra  \varphi \, \left(C(\bar \rho_s )-A(\bar \rho_s)\right) \big\ket \right| > \delta\right]=0,
\end{eqnarray*}
where $\bar  \rho$ is the unique weak solution of (\ref{limitehydro}).
\label{limitehydroK}
\end{thm}
Theorems \ref{thm: limitehydro}, \ref{limitehydroK} can be deduced from the methods used to prove the large deviations
so that their derivation is omitted (see also  \cite{DFL, BDGJL2}).

\bigskip

Before stating the large deviation principle, we need more notation. 
Let $G$ and $H$ be smooth functions in $[0,T] \times [-1,1]$.
For a given trajectory $(\rho,Q,K)$ in $\mathcal{E}$, we set
\begin{eqnarray}
J_{G,H} \left(\rho,Q,K\right)= J_H^1(\rho,Q) + J_G^2(\rho,K)
\label{eq: J(G,H)}
\end{eqnarray}
with 
\begin{align}
\begin{split}
J_H^1(\rho,Q)=&\left<{Q_T}\nabla H_T\right>-\int_0^T \left<{Q_s}\frac{d}{ds}\nabla H_s\right>ds
-\frac{1}{2}\int_0^T\left<\rho_s\Delta H_s\right>ds\\
& -\frac{1}{2}
\int_0^T\left<\sigma(\rho_s)(\nabla H_s)^2\right>ds
+\frac{1}{2}\bar \gr_+\int_0^T \nabla H(s,1) \, ds
-\frac{1}{2}\bar \gr_-\int_0^T \nabla H(s,-1)\, ds,
\end{split}
\label{J1}
\end{align}
where $\sigma(u)=u(1-u)$ is defined in \eqref{eq: conductivity} and
\begin{align}
\begin{split}
J_G^2(\rho,K)=\left<K_TG_T\right>-\int_0^T\left<K_s\frac{d}{ds}G_s\right>ds
-\int_0^T  \, ds \, \left< C(\rho_s)(e^{G_s}-1) +  A(\rho_s)(e^{-G_s}-1)\right> \, ,
\end{split}\label{J2}
\end{align}
where $A$ and $C$ were introduced in \eqref{CetA}.

The first functional is related to the contribution of the 
conservative currents and the second one to the non-conservative currents
(see theorem \ref{lowerreg}). We define
\begin{equation*}
J\left(\rho,Q,K\right) =\sup_{G,H}J_{G,H}\left(\rho,Q,K\right) 
=\sup_{H\in C^{1,2}}J_H^1(\rho,Q)+\sup_{G\in C^{1,0}}J_G^2(\rho,K) \, ,
\end{equation*} 
where the supremum is taken on regular functions $G$ and $H$.
Note that the functions $G$ and $H$ can take arbitrary (finite) values at the boundaries.

\medskip

Define $\rond{A}$ as the set of trajectories $(\rho,Q,K)$ satisfying the following two conditions:
\begin{itemize}
\item {\it Conservation law.} For all test function $\varphi \in C^1([-1,1])$ vanishing at the boundaries
\begin{equation}
\bra \rho_t \, \varphi \ket - \bra \rho_0 \varphi \ket =
\bra Q_t \, \partial_x \varphi \ket + \bra K_t \, \varphi \ket,\qquad Q_0=0,\ K_0=0.
\label{relation}
\end{equation}
\item {\it Energy condition.} The energy $\cQ(\gr)$ of the density trajectory is finite with  
\begin{eqnarray}
\cQ (\gr) = \sup_{\gp} \left\{
\int_0^T dt \int_{-1}^1 dx \; \gr(t,x) \nabla \gp(t,x) - 
\frac{1}{2} \int_0^T dt \int_{-1}^1 dx \, 
\gp(t,x)^2 \right.  \nonumber\\
\qquad \qquad \left.
- \int_0^T dt (\bar \gr_+ \, \gp(t,1) - \bar \gr_- \, \gp(t,-1))
\right\} \, ,
\label{eq: energy}
\end{eqnarray}
and the supremum is taken over smooth functions $\gp$ in $[0,T] \times [-1,1]$.
\end{itemize}

For smooth trajectories the conservation law reduces to  
$$
\partial_t \rho +\partial_x \partial_t Q -\partial_t K=0 \, ,
$$ 
where $ \partial_t Q, \partial_t K$ are the instantaneous currents and the energy condition reads
\begin{eqnarray*}
\cQ (\gr) = \frac{1}{2} \int_0^T dt \int_{-1}^1 dx \; \big( \nabla \gr(t,x) \big)^2
< \infty \, .
\end{eqnarray*}
The energy condition was introduced in \cite{QRV,BLM,FLM} to control the approximation procedure in the derivation of the 
large deviation lower bound (see Theorem \ref{lower}).

\medskip

Finally we define the dynamical rate function 
\begin{equation}
I_0(\rho,Q,K)=\left\{
\begin{array}{ll}
J(\rho,Q,K) & \textrm{if} \ \ 
(\rho,Q,K)\in\rond{A} \, , \\
+ \infty& \textrm{otherwise} \, .
\end{array}
\right. 
\label{I0}
\end{equation} 
To take into account the large deviations of the initial measure $\nu_\gamma^N$, we introduce for any function $m : [-1,1] \to [0,1]$
$$
h_\gamma(m)= \left \bra m\log\dfrac{m}{\gamma}\right \ket +\left \bra (1-m)\log\dfrac{1-m}{1-\gamma}\right \ket \, ,
$$
and
$$
I_\gamma (\mu) =
\left\{\begin{array}{ll}
h_\gamma(m)&\textrm{ if } \mu(dx) =m (x) \, dx \, ,\\
+\infty&\textrm{otherwise} \, .
\end{array} 
\right.
$$

The rate function is the sum of the dynamical deviation cost from the initial measure and the
deviation cost from the hydrodynamic trajectory 
$$
I(\rho,Q,K)=I_0(\rho,Q,K)+I_\gamma(\rho_0).
$$

From now, the initial density profile $\gamma$ is a given smooth function in $C_e([-1,1])$ equal to $\bar \gr_\pm$ at the boundaries.
We  state the large deviation Theorems.
\begin{thm}
For all closed set $\rond{F}\in\rond{E}$
$$
\limsup_{N\to\infty}\dfrac{1}{N}\log \p^N_\gamma\left[(\rho^N_t,Q^N_t,K^N_t)\in \rond{F}\right]
 \leqslant -\inf_{(\rho,Q,K)\in \rond{F}} I(\rho,Q,K).
$$
\label{upper}
\end{thm}

\medskip

We first state the lower bound for regular trajectories.
\begin{thm}
Let $(\rho,Q,K)$ be a regular trajectory. Then the large deviation functional $I_0$ has an explicit form
\begin{equation}
\label{eq: functional}
I_0( \gr, Q, K)
= \int_0^T  dt \int_0^1 dx \;
\left\{ \frac{\big( \dQ(x,t) + \partial_x \gr (x,t) \big)^2}{2 \gs \big( \gr (x,t)\big)}
+ \Phi \Big( \gr(x,t) , \dK(x,t) \Big) \right\} \, , 
\end{equation} 
with 
\begin{equation}
\label{eq: dissipative}
\Phi (\gr, \gk) = C(\gr) + A(\gr) - \sqrt{\gk^2 + 4  A(\gr) C(\gr)}
+ \gk \log \left( \frac{\sqrt{\gk^2 + 4  A(\gr) C(\gr)} + \gk}{2 C(\gr)} \right) \, .
\end{equation} 
If $C(\gr)=0$, then $\Phi$ becomes 
\begin{eqnarray*}
\Phi (\gr, \gk) = 
\begin{cases}
A(\gr) + \gk - \gk \log \left( \frac{- \gk }{ A(\gr)}\right) , \qquad  & {\rm if} \quad  \gk \leq 0,\\
\infty, \qquad  & {\rm if} \quad  \gk >0 \, .
\end{cases}
\end{eqnarray*} 
A similar formula holds if $A(\gr)=0$.

For any open set $\rond{O}\in\rond{E}$ containing the regular trajectory $(\rho,Q,K)$
$$
\liminf_{N\to\infty}\dfrac{1}{N}\log \p^N_\gamma\left[(\rho^N_t,Q^N_t,K^N_t) \in
\rond{O}\right] \geqslant -  I(\rho,Q,K).
$$
\label{lowerreg}
\end{thm}

\begin{rem}
The first contribution to \eqref{eq: functional} comes from the difference between the
instantaneous empirical current ${\partial_t Q}$ and the canonical
instantaneous current associated to $\rho$ which is $-\frac{1}{2}\nabla\rho$. 
This term has already been analyzed in the conservative dynamics \cite{BD, BDGJL2}.
The second term in \eqref{eq: functional} should be interpreted as the large deviation functional associated to 
Poisson processes with parameters $C(\rho_t)$ and $A(\rho_t)$.
\end{rem}

In order to derive the lower bound for general trajectories,
we introduce two technical assumptions on the rates \eqref{CetA}:

\begin{hyp}[\textbf{L1}]
The rate $A$ (resp $C$) is either concave and positive on $]0,1[$ or uniformly equal to zero.
\end{hyp}

\begin{hyp}[\textbf{L2}]
The functions $A$ and $C$ are monotonous and
\begin{eqnarray}
\label{eq: L2}
\forall z \in [0,1], \qquad A^\prime (z) \geq 0, \qquad C^\prime (z) \leq 0 \, . 
\end{eqnarray}
\end{hyp}

Then
\begin{thm}
\label{thm : lower general}
Assume $(\textbf{L1})$ and $(\textbf{L2})$, then for all open set $\rond{O}\in\rond{E}$,
$$
\liminf_{N\to\infty}\dfrac{1}{N}
\log \p^N_\gamma \left[(\rho^N_t,Q^N_t,K^N_t) \in \rond{O}\right]
 \geqslant -\inf_{(\rho,Q,K) \in \rond{O}} I(\rho,Q,K) \, .
$$
\label{lower}
\end{thm}

The concavity assumption $(\textbf{L1})$ has been introduced in \cite{JLV}.
As we shall see in Section \ref{sec: Approximation for general trajectories},
it  simplifies the proof of the lower bound, however it is mainly technical and Theorem \ref{lower} should be valid without assumption $(\textbf{L1})$.
We refer to \cite{QRV, BLM, FLM} for further results on this generalization in the case of conservative dynamics.
Assumption {\bf (L2)} will be used in the Appendix only to ensure the uniqueness of the weak solutions for singular perturbations of the hydrodynamic equation (\ref{limitehydro}).

\section{Modified dynamics and local equilibrium}
\label{sec: A super-exponential replacement lemma}

Local equilibrium lies at the heart of the hydrodynamic limit theory and it states that during the time evolution the local measure remains close to an equilibrium measure with a varying density. 
In this section, we state, in our framework, a strong form of local equilibrium which will be useful for the derivation of the hydrodynamic large deviations. The proofs are omitted as they follow the scheme introduced in \cite{KOV,KL,BDGJL1}.


\subsection{The modified dynamics}

We first  define a modification of the
original process (\ref{generateur}) which will be used to derive the large deviations.
For regular functions $G$ and $H$ on $[0,T]\times[-1,1]$, denote $L_N^{G_t,H_t}$ the time dependent
generator given by 
\begin{eqnarray}
\label{eq: modified dynamics}
L_N^{G_t ,H_t} =
\frac{N^2}{2}L_{0,N}^{H_t}+L_{1,N}^{G_t}+\frac{N^2}{2}L_{+,N}+\frac{N^2}{2}L_{-,N} \, ,
\end{eqnarray}
with
\begin{eqnarray*}
L_{0,N}^{H_t} f(\eta)&=&\sum_{x=-N}^{N-1}\left[f(\eta^{x, x+1})-f(\eta)\right]
\exp \Big[ (\eta(x)-\eta(x+1)) \Big(H(t,\frac{x+1}{N})-H(t,\frac{x}{N}) \Big) \Big] \, , \\
L_{1,N}^{G_t}f(\eta)
&=&\sum_{x=-N}^Nc(x,\eta) \left[\eta(x)
e^{-G(t,x/N)}+(1-\eta(x))e^{G(t,x/N)}\right] \left[f(\eta^x)-f(\eta)\right].
\end{eqnarray*}
The modified dynamics induces a weak drift in the conservative dynamics.
For large $N$, a particle jumps from $x$ to $x\pm1$ at rate
$\dfrac{1}{2}\left(1\pm\dfrac{1}{N} \partial_x
H\left(t,\frac{x}{N}\right)\right)$. $L_{1,N}^{G_t}$ is the generator
of a  non-conservative dynamics for which the
intensity of creation and annihilation varies in time and space according to $G$.
Finally, let $\p^N_{\gamma,G,H}$ be the probability measure associated to
the process with initial measure $\nu_\gamma^N$ and generator $L_N^{G_t ,H_t }$. 
We stress the fact that the reservoir dynamics are unchanged.
As in Theorem \ref{thm: limitehydro}, one can show that the modified dynamics follows the hydrodynamic limit
equation 
\begin{equation}
\label{eq: limit HG hydro}
\partial_t \rho (t,x) =\dfrac{1}{2}\Delta \rho (t,x) -\nabla \left( \gs \big( \rho(t,x) \big) \, \nabla H (t,x) \right)
+ C(\rho(t,x))e^{G(t,x)} - A(\rho(t,x))e^{-G(t,x)} \, .
\end{equation}

\subsection{Local equilibrium}
\label{sectionremplace}

We set $\Lambda_l^{(x)}=\{y\in\{-N,\dots,N\}\ \textrm{such that }|y-x|\leqslant l\}$
and define the local density as
\begin{equation}
\displaystyle{\bar \eta^l (x)= \frac{1}{|\Lambda_l^{(x)}|}\sum_{y\in\Lambda_l^{(x)}}\eta(y)} \, ,
\end{equation}
where $|\Lambda_l^{(x)}|$ stands for the number of sites in $\Lambda_l^{(x)}$.
Let $\psi$ be a cylindric function with support in $\{-R, \dots, R\}$.
Given $\varphi$ a regular function on $[0,T] \times [-1,1]$ and $\delta, \gep>0$, 
we define the set $B_{\delta,\gep,\varphi}(\psi)$ on the trajectories $\{\eta_s\}_{s \leq T}$ as 
\begin{equation}
B_{\delta,\gep,\varphi} (\psi)=\left\{
\{\eta_s\}_{s \leq T}, \qquad
\left|\int_0^T \, ds \, \frac{1}{2(N - R)}
\sum_{x = -N +R}^{N-R} \varphi \left( s, \frac{x}{N} \right) 
\left[\psi(\tau_x\eta_s)- \nu_{\overline{\eta}_s^{\epsilon N}(x)} (\psi)\right]\right|\leqslant\delta\right\} \, ,
\label{eq: B delta}
\end{equation}
where 
$\tau_x \eta$ is the configuration $\eta$ shifted by $x$.

\medskip

The local equilibrium property also holds for the modified dynamics \eqref{eq: modified dynamics}.
\begin{thm}
Given $\varphi$, $\psi$ and $\gd >0$, the trajectories concentrate super-exponentially fast on the set $B_{\delta,\gep,\varphi} (\psi)$
\begin{equation*}
\limsup_{\gep \to 0}
\limsup_{N\to\infty} \frac{1}{N}\log 
\p^N_{\gamma,G,H} \Big( B_{\delta,\gep,\varphi} (\psi)^c \Big)=-\infty \, .
\end{equation*} 
Moreover, the reservoirs impose local equilibrium at the boundaries with the densities $ \bar \rho_+,  \bar \rho_-$.
For any continuous function $\Phi$ in $[0,T]$
$$
\limsup_{N\to\infty}\dfrac{1}{N}\log \p^N_{\gamma,G,H}
\left[\Big|\int_0^T \, ds \,  \Phi(s)\big(\eta_s(\pm N) - \bar \rho_{\pm}\big)\Big|\geqslant\delta\right]=-\infty.
$$
\label{superexp}
\end{thm}
The derivation of Theorem \ref{superexp} follows from the bounds on entropy production \cite{KOV} which can be adapted to control the boundary terms as in \cite{BDGJL1}.

\section{Large deviation upper bound}
\label{sec: Large deviations upper bound}

The derivation of Theorem \ref{upper} is split into several steps. 
First an upper bound with the rate function $J$ \eqref{I0} is derived for compact sets, then for closed sets.
Finally, we prove that the rate function is infinite for the  trajectories which do not belong to the set $\rond{A}$ introduced in (\ref{relation}), (\ref{eq: energy}).

\subsection{The upper bound for compact sets}
\label{sectionderive}

In order to compare the original dynamics starting from the initial profile $\gamma$ to a modified dynamics 
with regular drifts $G,H$ \eqref{eq: modified dynamics} starting from the initial profile $\omega$,
we compute the Radon-Nikodym derivative \cite{KL}
\begin{eqnarray}
&& \dfrac{d\p^N_{\omega,G,H}}{d\p^N_\gamma}
=
\exp \bigg[N\Big\{
\bra Q_T^N \,\nabla H_T \ket - \int_0^T \, ds \, \bra Q_s^N \,\frac{d}{ds}
\nabla H_s \ket 
 +  \bra K_T^N \, G_T \ket - \int_0^T \, ds \, \bra K_s^N \,\frac{d}{ds}G_s \ket 
\nonumber \\
&& \qquad \qquad \qquad \qquad 
+O_H(\frac{1}{N}) \Big\} -\int_0^T R_N(s) ds +h_{\gamma,\omega}(\gr^N_0)  \bigg],
\label{radonnikodym}
\end{eqnarray} 
where
$h_{\gamma,\omega}(\rho) = \left \bra \rho \,\log(\omega/\gamma)\right \ket 
+ \left \bra (1-\rho) \, \log((1-\omega)/(1-\gamma))\right \ket$ 
and
\begin{eqnarray}
R_N(s)&=&\dfrac{N^2}{2}\Bigg(\sum_{x=-N}^{N-1}
\eta_s(x)(1-\eta_s(x+1)) 
\left[ \exp \left( H\left(s,\frac{x+1}{N}\right) - H\left(s,\frac{x}{N}\right) \right) -1\right] \nonumber\\
&&+\eta_s(x+1)(1-\eta_s(x)) \left[\exp \left( H\left(s,\frac{x}{N}\right)-H\left(s,\frac{x+1}{N} \right) \right)-1\right]
\Bigg)\nonumber\\
&&+\sum_{x=-N}^Nc(x,\eta_s)\left[ \exp \left(\left(1-2\eta_s(x)\right)
G(s,\frac{x}{N}) \right) -1\right].
\label{densite}
\end{eqnarray}
expanding the exponential and summing by parts, we get 
\begin{eqnarray*}
R_N(s)&=&
\dfrac{1}{2}\Bigg\{\sum_{x=-N}^{N-1}
\bigg[\eta_s(x) \Delta H \left(s,\frac{x}{N}\right) \\
&& + \dfrac{1}{4} \left(\nabla 
H \left(s, \frac{x}{N}\right)\right)^2 
\left[\eta_s(x) (1- \eta_s(x+1)) +\eta_s(x+1)(1 -\eta_s(x)) \right] \bigg] \Bigg\}\\
&& 
+\dfrac{1}{2}  \Big[ \eta_s(N)\nabla H(s,1)-\eta_s(-N)\nabla H(s,-1) \Big] \\
&&
+\sum_{x=-N}^Nc(x,\eta_s)\left[\exp \left( \left(1-2\eta_s(x)\right)G(s,\frac{x}{N}) \right) - 1\right]
+O(1) \, .
\end{eqnarray*}
Thanks to the local equilibrium, the microsopic expressions in $R_N(s)$ 
can be replaced by their averages. 
We set
\begin{eqnarray}
B_{\delta,\gep} = B_{\delta,\gep,(\nabla H)^2} \big( \eta_0 (1 - \eta_0) \big) 
\bigcap B_{\delta,\gep,({\rm e}^{G}-1)} \big( c(0,\eta) \eta_0 \big)
 \bigcap B_{\delta,\gep,({\rm e}^{-G}-1)} \big( c(0,\eta) (1-\eta_0) \big) 
 \bigcap B^\prime_{\gd,\pm} \, ,
\nonumber \\
\label{eq: B delata gep}
\end{eqnarray}
where the sets $B_{\delta,\gep,\varphi} (\cdot)$ were introduced in \eqref{eq: B delta}
and 
$$
B^\prime_{\gd,\pm} = \left\{ \Big|\int_0^T \, ds \,  \nabla H(s,\pm1) \, \big(\eta_s(\pm N)- \bar \rho_\pm \big)\Big| \leqslant \delta
\right\} \, .
$$

The super-exponential replacement Theorem \ref{superexp} implies
$$
\forall \delta>0,\qquad
\limsup_{\epsilon\to 0} \limsup_{N\to\infty} \dfrac{1}{N} \log 
\p^N_\gamma (B_{\delta,\gep}^c)=-\infty.
$$
For trajectories in $B_{\delta,\gep}$, the Radon-Nikodym derivative can be approximated as follows 
$$
\dfrac{d\p^N_{\omega,G,H}}{d\p^N_\gamma}=\exp\left[N\left(\J
(\rho^N,Q^N,K^N)+h_{\gamma,\omega}(\rho^N_0)+O_{G,H}(\frac{1}{N})+O(\delta)\right)\right] \, ,
$$
where we used the functional
\begin{equation}
\begin{array}{rcl}
\J(\rho,Q,K) &=& \bra Q_T \, \nabla H_T \ket 
- \int_0^T \, ds \, \bra Q_s \,\frac{d}{ds}\nabla H_s \ket  + \bra K_T \,G_T \ket 
-\int_0^T \, ds \,  \bra K_s \,\frac{d}{ds}G_s \ket  \\
&& -\frac{1}{2} \int_0^T \, ds \, \bra \rho_s \, \Delta H_s 
+  \sigma(\rho_s*i_{\epsilon}){\left(\nabla H_s\right)^2} \ket \\
&&-\int_0^T \, ds \, 
\left\bra C(\rho_s*i_{\epsilon})(e^{G_s}-1) + A(\rho_s*i_{\epsilon})(e^{-G_s}-1) \right \ket \\
&& + \frac{1}{2}\bar \gr_+\int_0^Tds\nabla H(s,1)-\frac{1}{2}\bar \gr_-\int_0^Tds\nabla H(s,-1)\, ,
\end{array}
\label{Jepsilon}
\end{equation}
with $A, C$ and $\gs$ as in (\ref{CetA}), \eqref{eq: conductivity}.
The function $i_\epsilon$ is an approximation of unity
\begin{equation*}
\rho_s^N*i_{\epsilon}=\overline{\eta}_s^{\epsilon N}.
\end{equation*}

\medskip

For any $\mathcal{O}$ open set in $\mathcal{E}$, we can then deduce the large deviation upper bound
\begin{eqnarray*}
\limsup_{N\to\infty} \dfrac{1}{N} \p_\gamma^N(\rond{O})
&\leqslant&
\max \bigg[\limsup_{N\to\infty}  \dfrac{1}{N}\log \p^N_\gamma \big( (\rho^N ,Q^N ,K^N ) \in \rond{O}\cap B_{\delta,\gep} \big) 
\, ; \,\\
&& \qquad \qquad \limsup_{N\to\infty}\dfrac{1}{N}\log \p^N_\gamma (B_{\delta,\gep}^c )\bigg]\\
&\leqslant&
\limsup_{N\to\infty} \dfrac{1}{N}\log
\p^N_{\omega,G,H} \left[\frac{d \p^N_\gamma}{d \p^N_{\omega,G,H}}
\, \mathbf{1}_{\{ \rond{O} \cap B_{\delta,\gep} \}} (\rho^N,Q^N,K^N)\right]\\
&\leqslant & 
\limsup_{N\to\infty}
\dfrac{1}{N}\log \p^N_{\omega,G,H} \left[e^{-N(\J ((\rho^N,Q^N,K^N))+h_{\gamma,\omega}(\rho^N_0)+O_{G,H}( \frac{1}{N})+O(\delta))} \,\mathbf{1}_\rond{O}\right]\\
&\leqslant&
\sup_{(\rho,Q,K)\in\rond{O}}
\; \Big\{ -\J(\rho,Q,K)-h_{\gamma,\omega}(\rho_0) \Big\}
+O(\delta).
\end{eqnarray*}
This is true for any $(\epsilon,G,H,\omega)$ and any $\delta>0$ so finally
\begin{equation}
\limsup_{N\to\infty} \dfrac{1}{N} \p^N_\gamma (\rond{O})
\leqslant
\inf_{\epsilon,G,H,\omega}\sup_{(\rho,Q,K)\in\rond{O}}
\; \Big\{-\J(\rho,Q,K)-h_{\gamma,\omega}(\rho_0) \Big\}.
\label{eq: upper ouvert}
\end{equation}

The previous bound can then be extended to any compact set $\cK$ by using a finite covering with open sets (see \cite{KL})
\begin{eqnarray*}
 \limsup_{N\to\infty} \dfrac{1}{N}\log \p^N_\gamma (\rond{K})
\leqslant - \sup_{(\rho,Q,K)\in\rond{K}} \;  \sup_{G,H} \{ J_{G,H} (\rho,Q,K) \} = - J (\rho,Q,K) .
\end{eqnarray*}

\subsection{The upper bound for closed sets} 
\label{fermes}

We are going to prove the exponential tightness, i.e. to exhibit a sequence $\{ \rond{K}_n \}$ of compact sets in $\rond{E}$ such that for any $n$
\begin{eqnarray}
\label{eq: exp tight}
\limsup_{N\to\infty} \dfrac{1}{N}\log \p^N_\gamma (\rond{K}_n^c) \leqslant -n.
\end{eqnarray}
The large deviation upper bound will then follow for general closed sets $F$ of $\rond{E}$ by noticing
\begin{eqnarray*}
 \limsup_{N\to\infty} \; \dfrac{1}{N} \log \p^N_\gamma (F) 
& \leqslant & \limsup_{N\to\infty} \; \dfrac{1}{N}\log\left[ \p^N_\gamma
\left(F\cap\rond{K}_n\right)+\p^N_\gamma (\rond{K}_n^c)\right]\\
&\leqslant& 
\max\left[\sup_{(\rho,Q,K) \in F\cap\rond{K}_n} \; J (\rho,Q,K)  \;  ;\;  -n\right].
\end{eqnarray*}
Letting $n$ go to $\infty$ completes the upper bound for closed sets.

\medskip

In order to build a sequence of compact sets, we need to check first  that the measures concentrate on equi-continuous trajectories
\begin{lem} 
\label{lemcontinuite}
Let $\phi$ be a $C^1$ function on $[-1,1]$. Then we have for all $\epsilon>0$,
\begin{eqnarray}
\lim_{\delta\to 0}\limsup_{N\to\infty}\dfrac{1}{N}
\log \p^N_\gamma
\left[\sup_{|t-s|\leqslant\delta}
\big|\left<\rho_t^N \, \phi \right>
- \left<\rho_s^N \, \phi\right> \big| > \epsilon\right]
&=&-\infty \, ,
\label{rho}\\
 \lim_{\delta\to 0}\limsup_{N\to\infty}\dfrac{1}{N}\log \p_\gamma^N\left[\sup_{|t-s|\leqslant\delta}
\big|\left<Q_t^N \phi \right> - \left<Q_s^N \, \phi \right>\big| > \epsilon\right]
&=&-\infty \, ,
\label{W}\\
\lim_{\delta\to 0}\limsup_{N\to\infty}\dfrac{1}{N}\log
\p_\gamma^N \left[\sup_{|t-s|\leqslant\delta}
\big|\left<K_t^N \phi \right> - \left<K_s^N \, \phi\right>\big|>\epsilon\right]
&=&
-\infty \, ,
\label{K}
\end{eqnarray}
where the supremum is taken over $s,t$ in $[0,T]$.
\end{lem}
Second, we need estimates on the total variation norm of the empirical currents  
\begin{lem} 
\label{lem: compacite}
For any time $T >0$, one has
\begin{eqnarray}
\label{eq: tension K 0}
\lim_{a \to \infty} \lim_{N \to \infty} \; \frac{1}{N}
\log \p^N_\gamma  \left ( \sup_{0 \leq t \leq T} \; \frac{1}{N} \sum_{x = -N}^N \big| K_t^N (x)  \big| \geq a  
\right) = - \infty \, ,
\end{eqnarray}
\begin{eqnarray}
\label{eq: tension Q 0}
\lim_{a \to \infty} \lim_{N \to \infty} \; \frac{1}{N}
\log \p^N_\gamma  \left ( \sup_{0 \leq t \leq T} \;  \frac{1}{N^2} \sum_{x = -N}^N \big| Q_t^N (x)  \big| \geq a 
\right) = - \infty \, .
\end{eqnarray}
\end{lem} 
Should the currents be bounded (as the density), then Lemma \ref{lemcontinuite} would be enough to ensure the exponential tightness \cite{KL}.

\medskip

We postpone the derivation of the Lemmas and conclude the proof of the exponential tightness \eqref{eq: exp tight}.
We focus on  the conservative current as the same strategy applies to $\gr, K$.
Fix a sequence $\phi_l$ of $C^2$ functions on $[-1,1]$ dense in $C([-1,1])$. 
We set
$$
\rond{C}_{l,\delta,m}=\left\{Q \in D([0,T],  \cM )\ ;\qquad \sup_{|t-s|\leqslant \delta}\bigg|
\bra Q_t \, \phi_l \ket - \bra Q_s\, \phi_l \ket \bigg|\leqslant\dfrac{1}{m}\right\} \, .
$$ 
Using Lemma \ref{lemcontinuite}, we have,
$$
\forall l\geqslant 0,\ \forall m\geqslant 1, \forall n\geqslant 1, \exists \delta(n,m,l),
\qquad  
\p^N_\gamma ( Q^N \notin \rond{C}_{l,\delta(n,m,l),m})\leqslant
\exp (-Nnml) \, .
$$ 
We introduce also
$
\cC^\prime_n = \left\{ Q \in D([0,T],  \cM )\ ; \quad  \sup_{t \leqslant T} \big| Q_t \big| \leqslant k(n) \right\},  
$
where $|Q_t|$ stands for the  total variation norm of the measure $Q_t$ and $k(n)$ is chosen, according to
Lemma \ref{lem: compacite}, such that 
$$
\p^N_\gamma ( Q^N \notin \cC^\prime_n) 
\leqslant \exp ( -Nn) \, .
$$

Now consider 
\begin{eqnarray}
\label{eq: compact set}
\rond{K}_n=\bigcap_{l\geqslant0,\ m\geqslant1}\rond{C}_{l,\delta(n,m,l),m}
\; \bigcap \; \cC^\prime_n  \, .
\end{eqnarray}
Ascoli Theorem (see  \cite{EK} Theorem 6.3 page 123) implies that $\rond{K}_n$ is a compact set.
Combining the previous estimates we see that $\cK_n$ satisfies \eqref{eq: exp tight}.


\medskip

\begin{proof}[Proof of Lemma \ref{lemcontinuite}]
We start by proving  (\ref{W})  and follow the strategy of \cite{BDGJL2}. 
It is enough to show that the expression below goes to $-\infty$ as $\gd$ vanishes
$$
\max_{0\leqslant k\leqslant \frac{T}{\delta}}
\limsup_{N\to\infty}\dfrac{1}{N}
\log \p^N_\gamma \left[\sup_{k\delta\leqslant t\leqslant(k+1)\delta}
\left<Q_t^N \ \phi \right> - \left<Q_{k\delta}^N \, \phi \right> > \epsilon \right].
$$
Thanks to the computation of the Radon-Nikodym derivative of the modified
dynamics (\ref{densite}) (with $ \nabla H = \frac{1}{N} \phi , G=0$),
 we know from \eqref{radonnikodym} that for all $a>0$,
\begin{eqnarray*}
\rond{M}_T &=&\exp\Bigg\{N \, \bra Q_T^N \, a \phi \ket 
-\dfrac{1}{2} \int_0^T \, ds \, \sum_{x=-N}^{N-1} \Bigg(\eta_s(x)a \, \nabla \phi
\left(\frac{x}{N}\right)  \\
&& + \dfrac{1}{4}\left[a \phi \left(\frac{x}{N}\right)\right]^2
\left[\eta_s(x) (1 - \eta_s(x+1)) + \eta_s(x+1 ) (1 - \eta_s(x)) \right] \Bigg)\\
&& + N \big[ \eta_s(N) \phi(1)-\eta_s(-N)\phi(-1) \big] +O(1) \Bigg\}
\end{eqnarray*}
is a mean one positive martingale. One easily checks that the integral term
 is bounded above by $C_ \phi a^2N T$ (we will take the limit as $a$ goes to infinity),
  where $C_\phi $ is a constant depending only on $\phi$.
 Therefore, multiplying by $aN$, adding and substracting the integral term of the logarithm of the martingale and exponentiating,
\begin{equation}
\label{eq: martingale}
\p^N_\gamma \left[
\sup_{k\delta\leqslant t\leqslant(k+1)\delta}
\bra Q_t^N \,\phi \ket -\bra Q_{k\delta}^N \, \phi \ket > \epsilon
\right]
\leqslant \p^N_\gamma \left[\sup_{k\delta\leqslant t\leqslant(k+1)\delta}
\dfrac{\rond{M}_t}{\rond{M}_{\delta k}}>\exp\left\{\epsilon a N - C_\phi a^2N\delta\right\}\right].
\end{equation}
Taking $\delta$ small enough such that $C_\phi a^2N\delta<\frac{\epsilon a N}{2}$,
 then Doob's inequality implies that the last expression is bounded above
 by $\exp[- aN \epsilon / 4]$. Letting $a$ go to $\infty$ completes the proof.

For the non-conservative current, (\ref{K}) follows in the same way by using the martingale 
\begin{equation*}
\rond{N}_t=\exp\Bigg[ \bra a \phi \, K_T^N \ket - \int_0^T \, ds \, \sum_{x= -N}^N
c(x,\eta_s)\left[e^{a(1-2\eta_s(x)) \phi(x/N)}-1\right] \Bigg] \, .
\end{equation*}
The proof of (\ref{rho}) can be found in \cite{KL}.
\end{proof}

\medskip

\begin{proof}[Proof of Lemma \ref{lem: compacite}]
The probability of the event in \eqref{eq: tension K 0} can be estimated from above by the 
large deviations of $(2N+1)$ independent Poisson processes thus
\begin{eqnarray}
\label{eq: tension K}
\forall a >0, \qquad
\p^N_\gamma  \left ( \sup_{0 \leq t \leq T} \;  \frac{1}{N} \sum_{x = -N}^N \big| K_t^N (x)  \big| \geq a 
\right)
\leqslant  \exp(- N C_a ) \, , 
\end{eqnarray}
where $C_a$ goes to infinity as $a$ diverges. 

\medskip

We turn now to the bound \eqref{eq: tension Q 0} and show that 
\begin{eqnarray}
\label{eq: tension Q}
\forall a >0, \qquad
\p^N_\gamma  \left ( \sup_{0 \leq t \leq T} \;  \frac{1}{N^2} \sum_{x = -N}^N \big| Q_t^N (x)  \big| \geq a \right)
\leqslant  e^{- N C^\prime_a } \, , 
\end{eqnarray}
where $C^\prime_a$ goes to infinity as $a$ diverges.
To prove \eqref{eq: tension Q}, we first use a microscopic identity (which holds at any time)
\begin{eqnarray}
\label{eq: identite micro}
\forall x \in \{-N,N\}, \qquad
Q_t^{N}(x-1) - Q_t^{N}(x)  = \eta^N_t(x) - \eta^N_0(x) -  K^N_t(x)  \, .
\end{eqnarray}
Therefore \eqref{eq: tension K} implies that with probability at least $1- e^{- N C_a}$, the conservative current through the edge $(x,x+1)$ satisfies
\begin{eqnarray*}
\forall t \in [0,T], \qquad
| Q_t^{N}(x)|  \leq  |Q_t^{N}( -N)| + 2 N (1 + a ) \, .  
\end{eqnarray*}
From this, we get
\begin{eqnarray}
&& \p^N_\gamma  \left ( \sup_{0 \leq t \leq T} \;  \frac{1}{N^2} \sum_{x = -N}^{N-1} \big| Q_t^N (x)  \big| \geq 5 a  \right)
\leqslant  e^{- N C_a } + \p^N\left (  \sup_{0 \leq t \leq T} \;  \frac{\big| Q_t^N (-N)  \big|}{N}   \geq 3 a -2   \right)
\, ,  \nonumber \\
&& \leqslant  e^{- N C_a t} + \p^N_\gamma \left (  \sup_{0 \leq t \leq T} \;  \frac{ Q_t^N (-N) }{N}   \geq 3 a -2   \right)
+ \p^N_\gamma  \left ( \inf_{0 \leq t \leq T} \; \frac{Q_t^N (-N)  }{N}   \leq - 3 a  + 2  \right) \, . \nonumber \\
\label{eq: borne intermediaire}
\end{eqnarray}
We bound now the first term on the RHS (the second one can be bounded similarly by symmetry).
Using again the identity \eqref{eq: identite micro}, we get
\begin{eqnarray*}
&& \p^N_\gamma  \left ( \sup_{0 \leq t \leq T} \;  \frac{Q_t^N (-N)  }{ N}   \geq 3 a -2   \right) 
\leq
\p^N_\gamma \left ( \sup_{0 \leq t \leq T} \;  \frac{1}{ N^2} \sum_{x = -N}^{N-1}  Q_t^N (x)  \geq  a - 4  \right) +  e^{- N C_a t}
\, .
\end{eqnarray*}
The terms in the RHS can be estimated as in \eqref{eq: martingale} by reducing to a martingale estimate.
This completes \eqref{eq: tension Q}.
\end{proof}

\subsection{The set $\cA$}

To complete the derivation of the upper bound, we prove that the trajectories concentrate exponentially fast on the set $\cA$ introduced in (\ref{relation}), (\ref{eq: energy}).

\medskip

\noindent
{\it The conservation law.}

Let $C^2_0([-1,1])$ be the set of twice differentiable functions vanishing at the boundary.
\begin{lem}
 For any $\phi \in C^2_0([-1,1])$, we introduce
\begin{eqnarray*}
&& 
V_T(\rho^N,Q^N,K^N,\phi)
=
\dfrac{1}{N}\sum_{x=-N+1}^{N-1} \phi(\frac{x}{N}) \big( \eta_T(x)
- \eta_0(x) \big) 
-
\dfrac{1}{N} \nabla \phi(\frac{x}{N}) \, Q_T^{N}(x) 
+ \phi(\frac{x}{N}) K_T(x) \, .
\end{eqnarray*}
Then for any $\delta>0$,
$$
\limsup_{N\to\infty}\dfrac{1}{N}\log \p^N_{\gamma}\left[|V_T(\rho^N,Q^N,K^N,\phi)| > \delta\right]=-\infty.
$$
\label{lemrelation}
\end{lem}

\begin{proof}

For any site $x$ and times $s<t$, the following microscopic relation holds
$$
\eta^N_t(x) - \eta^N_s(x) = [Q_t^{N}(x-1) - Q_s^{N}(x-1)]
-[Q_t^{N}(x) - Q_s^N(x)] + K^N_t(x) - K^N_s(x) \, .
$$
Summing in $x$ and integrating by parts the $Q$-term gives
$$
V_T(\rho^N,Q^N,K^N,\phi) =
\dfrac{1}{N}\sum_{x=-N}^{N-1} W_x Q_T^{N}(x),
\qquad {\rm with} \quad 
W_x = \int_{x/N}^{(x+1)/N} \, du \, \big( \frac{x+1}{N} - u \big)
\phi'' (u) \, ,
$$
where we used that $\phi$ vanishes at the boundaries.
From this identity we get for any $a>0$
\begin{eqnarray}
\p^N_\gamma \left( V_T(\rho^N,Q^N,K^N,\phi) >\delta \right)
\leqslant e^{-aN\delta} \; \p^N_\gamma  
\left( \exp \left( a \sum_{x=-N}^{N-1} W_x \, Q_T^{N}(x) \right) 
\right) \, .
\label{eq: borne sup conservation}
\end{eqnarray}
We first note that $W_x$ is of the order $1/N^2$
and that uniformly in $x$,  $W_{x+1} - W_x$ is of the order $\epsilon_N/N^2$ where $\epsilon_N$ vanishes to 0 as $N$ goes to infinity.
Thus the identity (\ref{radonnikodym}) implies
$$
\dfrac{d\p^N_{\gamma,0,H}}{d\p^N_\gamma} = 
\exp\left( a \, \sum_{x=-N}^{N-1} W_x Q_T^{N}(x)
+ a \,  o_\phi\left( N\right) \right) \, ,
$$
where $H$ is such that $H((x+1)/N) - H(x/N) = W_x$. 
This function is a mean one martingale and \eqref{eq: borne sup conservation} leads to
$$
\dfrac{1}{N}\log \p^N_\gamma 
\Big( V_T(\rho^N,Q^N,K^N,\phi) >\delta \Big)
\leqslant 
-a \delta + a \, o_\phi \left( 1 \right).
$$ 
Letting $N$ and then $a$ go
to $\infty$ concludes the proof of the lemma.
\end{proof}

\medskip

Let $(\phi_k)$ be a dense sequence of functions in $C^2_0 \left([-1,1]\right)$ and define
$$\rond{A}_n^{\delta}=\left\{(\rho,Q,K)\in\rond{E} \textrm{ such that }
 \forall\  k\leqslant n,\qquad  |V_T(\rho,Q,K, \phi_k)| \leqslant \delta \right\}.
$$
The previous lemma gives that for all $n$ and $\gd>0$,
$$
\limsup_{N\to\infty}\dfrac{1}{N}
\log \p^N_\gamma \left[(\rho^N ,Q^N ,K^N )\in F\right]
\leqslant 
-\inf_{(\rho,Q,K)\in F\cap\rond{A}_n^\gd} J(\rho,Q,K).
$$ 
This is true for any $n$ and $(\phi_n)$ is dense. So letting $\delta$ go to
$0$, we can define $I_0=+\infty$ for the trajectories which do not
satisfy the relation (\ref{relation}).

\medskip

\noindent
{\it Energy condition.} 

\begin{lem}
\label{lem: energy}
For any smooth $\gp  : (0,T)\times [-1,1] \to \bbR$. There is a constant $c_0$ such that 
\begin{eqnarray}
\label{eq: borne exp H1}
\forall k>0, \qquad \limsup_{\gep \to 0} \limsup_{N \to \infty} 
\frac{1}{N} \log \p^N_\gamma \Big( \cQ_\gp(\gr^N) \geq k
\Big) \leq -k + c_0 T \, ,
\end{eqnarray}
where $\cQ_\gp$ is defined for some suitable constant $c$
\begin{eqnarray}
\label{eq: cQ}
&& \cQ_\gp(\gr) = 
\int_0^T dt \int_{-1}^1 dx \; \gr(t,x) \nabla \gp(t,x) - \int_0^T dt (\bar \gr_+ \, \gp(t,1) - \bar \gr_- \, \gp(t,-1))
\nonumber \\
&& \qquad \qquad - \frac{1}{2 c} \int_0^T dt \int_{-1}^1 dx \; \gp(t,x)^2
\, .
\end{eqnarray}
\end{lem}
The proof of this Lemma follows from \cite{BLM, FLM} and therefore is omitted.
By considering  a dense sequence of functions $\{ \gp_k\}$, one can deduce from \eqref{eq: borne exp H1} \cite{FLM} that the large deviations are infinite for the trajectories $\gr$ such that 
$\sup_\gp \cQ_\gp(\gr) = +\infty$.
Note that if  $\gr$ is such that $\sup_\gp \cQ_\gp(\gr) < + \infty$, then $\gr$ is in $\bbL^2([0,T],\bbH(]-1,1[))$ and 
Riesz representation theorem implies that there exists $\nabla \gr$ such that for any $\gp  : (0,T)\times [-1,1] \to \bbR$.
\begin{eqnarray}
\int_0^T  dt \int_{-1}^1 dx  \rho(t,x) \nabla
\varphi(t,x)
-\{\bar \gr_+\varphi(t,1)-\bar \gr_-\varphi(t,-1)\} = \int_0^T  dt \, \int_{-1}^1dx  \nabla \rho (t,x) \varphi(t,x) \, .
\nonumber\\
\label{eq: weak bd 00}
\end{eqnarray}


\section{Large deviations lower bound}
\label{sec: Large deviations lower bound}

The derivation of the lower bound is split into two parts following the general scheme for hydrodynamic large deviations \cite{KOV, KL}. First we derive the lower bound 
for regular trajectories. Then under assumptions (${\bf L1, L2}$), we prove that general trajectories can be approximated by regular trajectories.
A key feature in this approximation procedure is that the contribution of both currents decouple
\begin{equation}
I_0(\rho,Q,K)=I_1(\rho,Q)+I_2(\rho,K)=
\sup_{H \in C^{1,2}}J_H^1(\rho,Q)+ \sup_{G \in C^{1,0}} J_G^2(\rho,K),
\label{decomposition}
\end{equation} 
where $J_H^1$ and $J_G^2$ were defined in \eqref{J1} and \eqref{J2}.
This simplifies some steps in the approximation as both functionals can be analyzed independently
contrary to the study in \cite{JLV}.
Assumption ({\bf L1}) provides some convexity properties of the large deviation functional which simplify the proof. We follow closely some arguments of \cite{KL,BDGJL2}. Thus we will sketch the main steps of the proofs
and only detail the new aspects related to the non-conservative currents.

\subsection{Lower bound for regular trajectories}
\label{rateregular} 

In this section we derive Theorem \ref{lowerreg}.
Suppose $\rho, Q,K$ are regular in time and space and that for
 any time $t$ and $x$, $\rho(t,x)$ is bounded away from $0$ and $1$. 
To the trajectory $(\rho, Q,K)$, one can associate the functions
 $G(t,x)$ and  $H(t,x)$ satisfying
\begin{equation}
\left\{
\begin{array}{lcl}
{\partial_t Q} (t,x) &=& -\dfrac{1}{2}\nabla \rho(t,x) + \gs \big( \rho(t,x) \big) \nabla H(t,x) \, ,\\
{\partial_t K} (t,x) &=& C(\rho(t,x))e^{G(t,x)}-A(\rho(t,x))e^{-G(t,x)} \, .
\end{array} 
\right.
\label{ghmax}
\end{equation}
The pointwise existence of $G$ comes from the fact that the
polynomial
 $C X^2-\dot K X -A$ has one positive root for positive $C$ and $A$.
 Furthermore $H$ is well defined as long as $\gs(\gr) = \rho(1-\rho)\neq 0$.
We choose $H(t,-1)=0$ and the value of $H$ at $x=1$ is imposed by 
equation \eqref{ghmax}, contrary to the case of density large deviations where $H$ is equal to 0 at both boundaries
(see section \ref{sec: densiteGD}).

\medskip

\begin{lem}
\label{lem: regularfunctionnal}
For regular trajectories $(\rho, Q,K)$ the functional $I_0$ (\ref{I0})
is given by 
\begin{eqnarray}
I_0 (\rho,Q,K) = J_{G,H} (\rho,Q,K) \, ,
\label{regularfunctionnal}
\end{eqnarray}
with $G$ and $H$ as in (\ref{ghmax}).
This expression coincides with the explicit form of the functional \eqref{eq: functional}.
\end{lem}

\begin{proof}
As $(\rho, Q,K)$ is regular in time,  $J_{\tilde G, \tilde H}$ can be rewritten after integration by parts as
\begin{eqnarray}
\label{eq: J IPP}
J_{\tilde G,\tilde H}(\rho,Q,K)
&=& \int_0^T\left<\left[{\partial_s Q}_s+\frac{1}{2}
\nabla\rho_s-\frac{1}{2} \gs(\rho_s) \nabla \tilde H_s\right] \ \nabla
\tilde H_s\right> ds \\
&& + \int_0^T \left<{\partial_s K}_s \, \tilde G_s - C(\rho_s) (e^{\tilde G_s}-1)-  A(\rho_s) (e^{-\tilde G_s}-1)\right> \, ds \, , \nonumber
\end{eqnarray}
for any $(\tilde G, \tilde H)$ smooth functions.

For $G$ and $H$ given by (\ref{ghmax}), we are going to check that
\begin{eqnarray}
J_{G,H}(\rho,Q,K)\geqslant
\sup_{\tilde{G}, \tilde{H}}J_{\tilde{G},\tilde{H}}(\rho,Q,K).
\label{eq: J,H}
\end{eqnarray}
Indeed, take $\tilde{H}=H+F$ and any $\tilde{G}$, then it is easy to
see that
$$
J_{\tilde{G},\tilde{H}}(\rho,Q,K)=J_{\tilde{G},H}(\rho,Q,K)-\dfrac{1}{2}
\int_0^T \, ds \, \left< \gs \big( \rho_s \big)\ |\nabla F_s|^2\right>
\leqslant
J_{\tilde{G},H}(\rho,Q,K).
$$ 
Moreover, if $\tilde{G}=G+F$ and $\tilde{H}$ any regular function, then
from  the expression of  $J_G^2$  \eqref{J2} and the identity \eqref{ghmax} we get
\begin{align*}
\begin{split}
& J_{\tilde{G},\tilde{H}}(\rho,Q,K)=J_{G,\tilde{H}}(\rho,Q,K)
+\int_0^T \, ds \, \left<C(\rho_s)e^{G_s}\left(1+F_s-e^{F_s}\right)\right>\\
& \qquad \qquad +\int_0^T \, ds \, \left<A(\rho_s)e^{-G_s}\left(1-F_s-e^{-F_s}\right)\right> 
\leqslant J_{G,\tilde{H}}(\rho,Q,K) \, ,
\end{split}
\end{align*}
where we used that $\exp(x) -x -1 \geq 0$ for any $x \in \bbR$.
This completes the proof of \eqref{eq: J,H}.

\medskip

Finally for $G$ and $H$ given by (\ref{ghmax}), one can rewrite \eqref{eq: J IPP}
\begin{eqnarray*}
J_{G,H}(\rho,Q,K) 
&=& \dfrac{1}{2}\int_0^T \left<\dfrac{1}{\gs \big( \rho_s \big)}\ \big( {\partial_s Q}_s+
\frac{1}{2}\nabla\rho_s \big)^2\right>ds 
+ 
\int_0^T\left<C(\rho_s)\ \left(1-e^{G_s}+G_se^{G_s}\right)\right>ds
\\&& +\int_0^T\left<A(\rho_s)\ \left(1-e^{-G_s}-G_se^{-G_s}\right)\right> \, ds \,  .
\end{eqnarray*} 
This leads to the explicit form of the functional \eqref{eq: functional}.
\end{proof}

\begin{rem}
 For $G$ and $H$ as in (\ref{ghmax}), then the conservation equation
  $\partial_t \rho={\partial_t K}-\nabla {\partial_t Q}$ implies that $\gr$ obeys the hydrodynamic limit 
  \eqref{eq: limit HG hydro} of the modified dynamics $\p^N_{G,H}$ \eqref{eq: modified dynamics}.
\label{hydrodrift}
\end{rem}

\medskip

We introduce now the set of regular trajectories
\begin{definition}
\label{def: cS}
Denote by $\rond{S}$ the set  of trajectories $(\rho,Q,K)$ satisfying $I(\rho,Q,K)<\infty$
and such that 
\begin{itemize}
\item $\rho$ is bounded away from 0 and 1: there is $\gep >0$ such that $\gep< \gr(t,x) < 1-\gep$ for any $(t,x)$ in $[0,T] \times [-1,1]$.
\item There exists two regular functions $G$ and $H$ such that $(\rho,Q,K)$ is a weak solution (in the sense of \eqref{limitehydro general}) of
\begin{equation}
\left\{
\begin{array}{l}
\partial_t \rho (t,x) =\dfrac{1}{2}\Delta \rho(t,x)-\nabla\left( \gs \big( \rho(t,x) \big) \nabla H(t,x)\right) + C(\rho(t,x))e^{G(t,x)} - A(\rho(t,x))e^{-G(t,x)},\\
\rho(t, \pm1)=\bar \rho_\pm, \qquad \rho(t=0,x) = \gamma(x) \, ,
\end{array}
\right.
\label{S}
\end{equation}
\item For any smooth test function $\varphi$ in $C^{1,1} ([0,T] \times [-1,1])$
\begin{eqnarray}
\label{eq: weak Q 2}
\left<{Q_T} \, \varphi_T\right>-\int_0^T \left<{Q_s} \, \frac{d}{ds} \varphi_s\right>ds
&& = - \frac{1}{2}\bar \gr_+\int_0^T \varphi(s,1)\ ds 
+ \frac{1}{2}\bar \gr_-\int_0^T \varphi (s,-1) \ ds 
 \nonumber \\
&&  + \frac{1}{2}\int_0^T \left<\rho_s \, \nabla \varphi_s \right> \, ds +
\int_0^T\left<\sigma(\rho_s)\nabla H_s \,  \varphi_s \right>ds, 
\end{eqnarray}
\begin{eqnarray}
\label{eq: weak K 2}
\left<K_T \, \varphi_T \right> - \int_0^T \left<K_s\frac{d}{ds} \varphi_s\right> \, ds
 = \int_0^T  \, ds \, \left< \big[ C(\rho_s)e^{G_s} -  A(\rho_s)e^{-G_s} \big] \; \varphi_s \right>.  
\end{eqnarray}
\end{itemize}
\end{definition}


The same argument as in Lemma \ref{lem: regularfunctionnal} implies that for any trajectory $(\gr,Q,K)$ in $\cS$
\begin{align} 
\begin{split}
I_0(\rho,Q,K)&= J_{G,H} (\rho,Q,K)\\
&= \dfrac{1}{2}\int_0^T \left<\gs \big( \rho_s \big) \,  |\nabla H_s|^2\right>ds + \int_0^T\left<C(\rho_s)
\ \left(1-e^{G_s}+G_se^{G_s}\right)\right>ds
\\ &\quad +\int_0^T\left<A(\rho_s)\ \left(1-e^{-G_s}-G_se^{-G_s}\right)\right>ds \, .
\end{split}
\label{rem: regularsolution}
\end{align}


We prove now the lower bound for trajectories in  $\rond{S}$.
Theorem \ref{lowerreg} is a direct consequence of relation \eqref{rem: regularsolution}
and of the following Proposition.
\begin{prop} 
For all open set $\rond{O}\in\rond{E}$,
$$
\liminf_{N\to\infty} \; \dfrac{1}{N} \log \p^N_\gamma\left[ (\rho^N ,Q^N ,K^N) \in \rond{O}\right]
 \geqslant -\inf_{(\rho,Q,K)\in \rond{O}\bigcap\rond{S} } I(\rho,Q,K).
$$
\label{prop: S}
\end{prop}

\begin{proof}
Let  $(\rho,Q,K)$ be in $\rond{O} \bigcap \rond{S}$ and
 satisfying $I(\rho,Q,K)<\infty$. There is regular $(G,H)$ for which $\rho$ is a weak solution of (\ref{S}).
  Thanks to Lemma \ref{superexp}, there is $\epsilon >0$ such that 
  the trajectories concentrate in the set $B_{\delta,\gep}$ \eqref{eq: B delata gep}
\begin{equation}
\limsup_{N\to\infty} \dfrac{1}{N} \log
\p^N_{\gamma,G,H} (B_{\delta,\gep}^c) \leqslant -1.
\label{bonepsilon}
\end{equation}

The function $\J$ defined in (\ref{Jepsilon}) is continuous on
$\rond{E}$. Moreover, since $\rho_0$ is bounded away from 0 and 1, the function
$ \gp \to h_{\gamma,\rho_0}( \gp (0,\cdot))$ is continuous on $D([0,T], \cM_0)$. Let $\rond{V}\subset\rond{O}$
 be an open neighborhood of $(\rho,Q,K)$ such that
$$\forall (\gp,U,L)\in \rond{V}, 
\qquad \qquad \ |\J(\gp,U,L)-\J(\rho,Q,K)|<\delta,$$
$$
|h_{\gamma,\rho_0}(\gp_0)-h_{\gamma,\rho_0}(\rho_0)|=
|h_{\gamma,\rho_0}(\gp_0)-h_{\gamma}(\rho_0)|<\delta.
$$
Using the change of measure \eqref{radonnikodym}
\begin{eqnarray*}
 \p^N_\gamma (\rond{O}) 
& \geqslant & \p^N_\gamma(\rond{V})
\geqslant \p^N_\gamma (\rond{V}\cap {B_{\delta,\epsilon} }) 
\geqslant
\p^N_{\rho_0,G,H} \left(\dfrac{d \p^N_\gamma}{d \p^N_{\rho_0,G,H}}
\mathbf{1}_{\rond{V}\cap {B_{\delta,\epsilon} }}\right)\\
&\geqslant& \exp\left\{-N(\J(\rho,Q,K)+
h_{\gamma}(\rho_0)+O_{G,H}(N^{-1},\epsilon,\delta))\right\}
\p^N_{\rho_0,G,H} \left(\rond{V}\cap {B_{\delta,\epsilon} }\right) \, .
\end{eqnarray*}
Thanks to (\ref{bonepsilon}) and the hydrodynamical limit
for the perturbed process  \eqref{eq: limit HG hydro}
$$
\lim_{N\to\infty} \p^N_{\rho_0,G,H} \left(\rond{V}\cap {B_{\delta,\epsilon}    }\right)=1 \, ,
$$
where we used the uniqueness of the weak solution of (\ref{S}) (see the Appendix) to conclude that the
probability of $\rond{V}$ converges to $1$.
This leads to
$$
\liminf_{N \to \infty} \dfrac{1}{N}\log \p^N_\gamma (\rond{O})\geqslant -\J(\rho,Q,K)
-h_{\gamma,\rho_0}(\rho_0)+O_{G,H}(\delta,\epsilon).
$$ 
Letting $\epsilon\downarrow 0$ and $\delta\downarrow 0$, and since $(G,H)$
satisfy (\ref{S}), the Proposition is completed. 
\end{proof}

\subsection{Approximation for general trajectories}
\label{sec: Approximation for general trajectories}

To complete Theorem \ref{thm : lower general} for general open sets, it remains to prove that
\begin{lem}
\label{lem: I-regularite}
We assume ({\bf L1}, {\bf L2}).
For any $(\rho,Q,K)$ such that $I(\rho,Q,K) < \infty$ there is a sequence $(\rho^{(n)},Q^{(n)},K^{(n)})$ in $\cS$ converging 
weakly to $(\rho,Q,K)$ such that
\begin{eqnarray}
I (\rho,Q,K) = \lim_{n \to \infty} I (\rho^{(n)},Q^{(n)},K^{(n)})  \, .
\label{eq: lower bound regul}
\end{eqnarray}
\end{lem}

Lemma \ref{lem: I-regularite} implies  that
\begin{eqnarray}
\label{eq: approx inf S 0}
\inf_{(\rho,Q,K)\in \rond{O}\bigcap \cS }
I(\rho,Q,K)= \inf_{(\rho,Q,K)\in \rond{O}} I(\rho,Q,K).
\end{eqnarray}
Combining this identity and Proposition \ref{prop: S} proves Theorem \ref{thm : lower general}.

\subsubsection{Bounding the density away from 0 and 1}
\label{away}

We first approximate the density by trajectories bounded away from 0 and 1.
\begin{lem}
\label{lem: away}
Let $P=(\rho,Q,K)$ be a path such that $I(P)<+\infty$. There is  
$P_\gd =(\rho_\gd,Q_\gd,K_\gd)$ with density $\gr_\gd$ uniformly bounded away from 0 and 1 
which converges to $P$ and such that 
$$
I(P)=\lim_{\delta\to 0} I(P_\delta) \, .
$$
\end{lem}

\begin{proof}
Using Assumption (${\bf L1}$), we first establish a property of the functional $I_0$.
We use the decomposition \eqref{decomposition} of $I_0$.
$I_1$ is a convex functional of $(\gr,Q)$ as it is  the supremum of $J_H^1$ which are convex functionals of 
$(\gr,Q)$ (we used that $\sigma(\rho)=\rho(1-\rho)$ is concave).
$I_2$ is not convex, but we use a trick introduced in \cite{JLV} and decompose $I_2$ as
\begin{equation*}
I_2 (\rho,K) = \sup_{G \in C^{1,0}} \tilde{J}_G^2(\rho,K) + \int_0^T \, dt \, \bra C(\rho_t ) + A(\rho_t) \ket \, .
\end{equation*}
Since $A$ and $C$  are concave,  $\sup_{G} \tilde{J}_G^2(\rho,K)$ is a convex functional of $(\gr,K)$.
Thus  the large deviation functional can be decomposed into two terms
\begin{equation}
I_0 (\rho,Q,K) =
\tilde{I}_0 (\rho,Q,K) + \int_0^T \, dt \, \bra C(\rho_t ) + A(\rho_t) \ket \, ,
\label{tildei}
\end{equation}
and  $\tilde{I}_0$ is a convex functional of $(\rho,Q,K)$.
We deduce that  $\tilde{I}_0$  is lower semi-continuous for the weak topology.

\medskip

We turn now to the approximation procedure.
Let $\bar{P}=(\varphi,U,M)$ be the solution of the hydrodynamic equation \eqref{limitehydro}
$$
\left\{\begin{array}{lcl}
\partial_t \varphi (t,x)
&=& \frac{1}{2}\Delta\varphi (t,x)+ C(\varphi(t,x)) - A(\varphi(t,x))\\
\partial_t U (t,x) &=&-\frac{1}{2} \nabla \varphi (t,x)\\
\partial_t M (t,x) &=& C(\varphi (t,x))-A(\varphi (t,x))
\end{array}
\right.$$
with boundary conditions 
$$
\varphi_0 (0,x) =\gamma (x) \ , \ U(0,x) =0\ ,\ M(0,x) = 0,\ 
\varphi(t,-1)=\bar \gr_-,\ \varphi(t,1)=\bar \gr_+ \, .
$$
By construction $I (\bar{P})=0$.
We set $P_\delta = (1-\delta)P+\delta \bar{P}$ which has a density bounded away from $0$ and $1$. 
As $\tilde{I}_0 (\bar{P}) < \infty$, the convexity implies that $\tilde{I}_0 (P_\delta)\leqslant
(1-\delta) \tilde{I}_0 (P)+ \gd \tilde{I}_0 (\bar{P})$ so that
\begin{equation*}
\limsup_{\delta\to 0}
\tilde{I}_0(P_\delta)\leqslant \tilde{I}_0(P) \, .
\end{equation*} 
 As $P_\delta$ weakly converges to $P$, the lower semi-continuity of $\tilde I_0$ implies
$$
\tilde{I}_0(P)\leqslant \liminf_{\delta\to 0} \tilde{I}_0(P_\delta) \, .
$$
Thus  $\tilde{I}_0(P_\delta)$ converges to $\tilde{I}_0(P)$. 
Finally $h_\gamma$ and
$\rho \to \int_0^T \, dt \, \bra C(\rho_t)+A(\rho_t) \ket$ are continuous for $\| \cdot \|_\infty$. This completes the Lemma.
\end{proof}

\subsubsection{Time regularisation}
\label{sec: Time regularisation}

We will prove
\begin{lem}
For any path $P=(\rho,Q,K)$ such that $I(P)<\infty$,
there is a sequence regular in time $P_\gep=(\rho_\epsilon, Q_\epsilon,K_\epsilon)$ converging weakly to $(\rho,Q,K)$ 
such that  $I(P_\epsilon)$ converges to $I(P)$.
\label{timereg}
\end{lem}

In the following, only the regularity of $K$ is needed in
order to construct a  drift $G$ adapted to the non-conservative current \eqref{eq: drift G}. 
However, the regularizing sequence has to satisfy the conservation law  (\ref{relation}) so that  
$\rho, Q, K$ will be approximated simultaneously.

\begin{proof}
 
The proof is based on a time convolution and follows similar steps as in Lemma \ref{lem: away}.
We just recall the salient features of the proof (see \cite{KL} for further details).

Let $\psi_\epsilon$ be  a $C^\infty$ approximation of unity  such that $\psi_\epsilon=0$ outside
$[0,\epsilon]$ and $\int_0^\epsilon \psi_\epsilon =1$. To take the
convolution product of $(\rho,Q,K)$ with $\psi_\epsilon$, we have to
extend the path $(\rho,Q,K)$ beyond the time $T$. Let $r$ be the solution of
the hydrodynamic equation \eqref{limitehydro}
\begin{eqnarray*}
 \partial_s r_s &= \dfrac{1}{2}\Delta r_s+C(r_s)-A(r_s), \qquad
{\rm with} \qquad  r_0&=\rho_T\ .
\label{restsolution}
\end{eqnarray*}
 We set for $s>T$
\begin{equation*}
 \rho_s=r_{s-T}, 
 \qquad
Q_s=Q_T-\dfrac{1}{2}\int_T^s du \, \nabla \rho_u ,
\qquad K_s=K_T+\int_T^s du \, [C(\rho_u)-A(\rho_u)]\ .
\end{equation*}
Now we define 
\begin{eqnarray}
 \label{eq: regular path}
\begin{cases}
\rho_\epsilon(t,x) = \rho*\psi_\epsilon (t,x),\\
Q_\epsilon(t,x) = Q*\psi_\epsilon (t,x) - Q*\psi_\epsilon (0,x),\\
K_\epsilon(t,x) = K*\psi_\epsilon (t,x) - K*\psi_\epsilon (0,x) \, . 
\end{cases}
\end{eqnarray}
The path $(\rho_\epsilon, Q_\epsilon,K_\epsilon)$  satisfies the relation (\ref{relation}) and has initial currents equal to 0. We use the decomposition \eqref{decomposition} and approximate independently  the functionals $I_1$ and $I_2$ by using convexity properties 
deduced from Assumption ({\bf L1}).
\end{proof}

\subsubsection{Non regular drifts}
\label{sec: drift non conserved}

Thanks to Lemmas \ref{lem: away} and  \ref{timereg}, it is enough to consider a trajectory $(\gr,Q,K)$ regular in time with a density uniformly bounded away from 0 and 1. 
We are going to associate to $(\gr,Q,K)$ the drifts $H,G$ as for the trajectories in $\cS$ introduced in Definition \ref{def: cS}. 
Note that the drifts $H,G$ can be  non regular in space.

As ${\partial_s K}$ exists, the drift $G$ is defined (as for the regular
trajectories) as the solution of
\begin{equation}
\forall x \in [-1,1], \qquad
{\partial_s K}(s,x) = C(\rho(s,x))e^{G(s,x)} - A(\rho(s,x))e^{-G(s,x)}.
\label{eq: drift G}
\end{equation}
Note that $K,G$ solve the equation \eqref{eq: weak K 2} and as in the regular case
(\ref{rem: regularsolution})
\begin{equation}
I_2 (\rho,K)=\int_0^T\left<C(\rho_s)\
\left(1-e^{G_s}+G_se^{G_s}\right)\right> + \left<A(\rho_s)\
\left(1-e^{-G_s}-G_se^{-G_s}\right)\right>ds. \label{bonIdeux}
\end{equation}

The functional $I_1$ is the same as the functional of the SSEP (without Glauber rates).
Thus the drift of the conservative dynamics can be approximated thanks to the Riesz representation theorem as in \cite{KL, BDGJL2}.
For $(\rho,Q)$ such that $I_1(\rho,Q)<\infty$, there is $H$ in $\bbL^2([0,T],\bbH_1(]-1,1[) )$  such that
\begin{equation}
I_1(\rho,Q)= \int_0^T\int_{-1}^1 \dfrac{\sigma(\rho(s,x))}{2} \, (\nabla H(s,x))^2 \; dx\, ds \, .
\label{bonIun}
\end{equation}
Moreover \eqref{eq: weak Q 2} holds.
Note that the time regularity is not needed to derive \eqref{bonIun}.

\medskip

Finally, we check that $\gr$ is a weak solution of \eqref{limitehydro general} with the drifts $G,H$.
As the large deviation functional is finite, $\gr$ belongs to $\cA$ so that $\cQ(\gr)$ is finite and
\eqref{eq: weak bd 2} is satisfied. Combining the conservation law \eqref{relation}
with \eqref{eq: weak Q 2} and \eqref{eq: weak K 2} shows that \eqref{weaksoldrift 2} holds.

\subsubsection{Approximation paths in $\cS$}
\label{subsec: approx}

We approximate now $H,G$ by regularized drifts $H_n,G_n$ and prove that the associated sequence of paths $(\gr^{(n)}, Q^{(n)}, K^{(n)})$ in $\cS$ converges to $(\gr,Q,K)$.

Let $H_n$ be a sequence of $C^\infty$ functions converging to $H$ in $\bbL^2([0,T],\bbH_1(]-1,1[) )$.
Since $C$ and $A$ satisfy assumption $(\textbf{L1})$ and $\rho$ is uniformly bounded away from
$0$ and $1$ (Lemma \ref{lem: away}), $C$ and $A$ are either uniformly
bounded away from $0$, or uniformly equal to 0. 
We will consider only the case $A>0,C>0$ as the other case follows in the same way.
From  (\ref{bonIdeux}), we get that
$|G|e^{|G|}\in \bbL^1([0,T]\times[-1,1])$, so that there is $G_n$ such that
\begin{equation}
\int_0^T\int_{-1}^1|G_ne^{\pm G_n}-G e^{\pm G}| \xrightarrow[n\to\infty]{}0\qquad\textrm{    and    }
  \qquad\int_0^T\int_{-1}^1|e^{\pm G_n}- e^{\pm G}|\xrightarrow[n\to\infty]{}0.
\label{Gconvergence}
\end{equation}

Let $(\gr^{(n)},Q^{(n)},K^{(n)})$ be the sequence of paths in $\cS$ associated to the regular drifts $H_n, G_n$.
In particular $\rho^{(n)}$ is  the weak solution of
\begin{align}
\left\{\begin{array}{l}
\partial_t \rho^{(n)} =\dfrac{1}{2}
\Delta \rho^{(n)} -\nabla (\sigma(\rho^{(n)})\nabla H_n)  + C(\rho^{(n)})e^{G_n} - A(\rho^{(n)})e^{-G_n}\\
\rho^{(n)}(t,\pm 1)= \bar \rho_\pm, \qquad \rho^{(n)}(0,\cdot)=\rho (0,\cdot) \, ,
\end{array}\right.
\label{edp}
\end{align}
and $Q^{(n)}$ and $K^{(n)}$ are defined as
\begin{eqnarray}
\label{eq: weak approx n}
Q^{(n)} (t,x) &=& \int_0^t \big(-\dfrac{1}{2}\nabla \rho^{(n)} (s,x) + \sigma(\rho^{(n)}(s,x))\nabla H_n(s,x) \big) \, ds,\\
K^{(n)} (t,x) &=& \int_0^t \left(C(\rho^{(n)}(s,x))e^{G_n(s,x)}-A(\rho^{(n)}(s,x))
e^{-G_n(s,x)}\right)ds\ . \nonumber
\end{eqnarray}

From energy estimates (using similar bounds as in the Appendix \eqref{eq: borne d'energie}) we get  
\begin{equation}
\int_0^Tdt\int_{-1}^1dx(\nabla \rho^{(n)})^2  \leqslant  C,
\label{bornegradrhon}
\end{equation}
where $C$ is a constant independent of $n$. 
As $\rho^{(n)}$ and $\nabla \rho^{(n)}$ are bounded sequences in $\bbL^2$, they are tight in the weak topology. 
We want to check the uniqueness of the limiting points.
Let $R$ be a limit point of $\rho^{(n)}$. We will prove
that $R$ is a weak solution of \eqref{limitehydro general}.
From Section \ref{sec: drift non conserved}, we know that $\gr$ is also a weak solution of \eqref{limitehydro general}
so that the uniqueness of the weak solutions (see the Appendix) will imply that $\gr = R$.

\medskip

We want to take the limit in the weak formulation of (\ref{edp}). 
By construction, for any smooth function $\varphi$ on $[-1,1]$, we have
\begin{equation*}
 \int_0^tds\int_{-1}^1dx\rho^{(n)}(s,x)\nabla\varphi(x)=\left[\bar \gr_+\varphi(1)-\bar \gr_-\varphi(-1)\right]
-\int_0^tds\int_{-1}^1dx\nabla\rho^{(n)}(s,x)\varphi(x).
\end{equation*}
By weak convergence of the subsequence, one has
\begin{equation*}
 \int_0^tds\int_{-1}^1 \, dx \, R (s,x) \nabla \varphi(x) = \left[\bar \gr_+\varphi(1)-\bar \gr_-\varphi(-1)\right]
- \int_0^tds\int_{-1}^1 dx \nabla R(s,x)\varphi(x).
\end{equation*}
Furthermore \eqref{bornegradrhon} implies that the limit  $\nabla R$ is also in $\bbL^2$.
It remains to take the limit in $n$ in the equation 
\begin{multline}
 \int_{-1}^1\rho^{(n)}(t)\varphi=\int_{-1}^1 \gamma \varphi
 -\dfrac{1}{2}\int_0^t\int_{-1}^1\nabla\rho^{(n)}\nabla\varphi
+\int_0^t\int_{-1}^1\nabla \varphi\sigma(\rho^{(n)})\nabla
H_n\\+\int_0^t\int_{-1}^1C(\rho^{(n)})e^{G_n}\varphi
-\int_0^t\int_{-1}^1A(\rho^{(n)})e^{-G_n}\varphi \, .
\label{edpenN}
\end{multline} 
The difficulty is to treat the non linear terms.
We proceed term by term and start with the non-linearity $C$  in \eqref{edpenN} (the term in $A$ can be controlled in the same way). 
By \eqref{Gconvergence}, $\exp(G_n)$ converges to $\exp(G)$ in $\bbL^1$  thus it is enough to check that 
\begin{eqnarray}
\label{eq: decomposition 0}
\lim_{n \to \infty} \int_0^t\int_{-1}^1C(\rho^{(n)})e^{G}\varphi 
= \int_0^t\int_{-1}^1C(R)e^{G}\varphi \, .
\end{eqnarray}
We write
\begin{eqnarray}
\label{eq: decomposition}
&& \int_0^t\int_{-1}^1  \, \big( C(\rho^{(n)}) - C(R) \big) \varphi  \, e^{G}
= 
\int_0^t\int_{-1}^1  \, \big( C(\rho^{(n)}) - C(\rho^{(n)} *\iota_\delta ) \big) \varphi  \, e^{G}\\
&& \qquad + \int_0^t\int_{-1}^1  \, \big( C(\rho^{(n)} *\iota_\delta )  - C( R *\iota_\delta ) \big) \varphi  \, e^{G}
+ \int_0^t\int_{-1}^1  \, \big( C( R *\iota_\delta ) - C(R) \big) \varphi  \, e^{G} \, . \nonumber
\end{eqnarray}
As $\gd$ goes to 0, $R *\iota_\delta$ converges to $R$. 
The convergence in the weak topology implies that $\rho^{(n)} *\iota_\delta$ converges a.s. to $R *\iota_\delta$ when $n$ goes to infinity. 
Suppose that 
\begin{eqnarray}
\lim_{\gd \to 0} \sup_n \left| \int_0^t\int_{-1}^1 \,  \Big( C(\rho^{(n)}) -  C(\rho^{(n)} * \iota_\delta ) \Big)  e^{G}\varphi
\right| = 0 \, ,
\label{eq: converge unif}
\end{eqnarray} 
then  choosing $\gd$ small and then $n$ large, the convergence in \eqref{eq: decomposition 0} follows.


To derive \eqref{eq: converge unif}, we first note that for $x \in ]-1 +\gd, 1-\gd[$
\begin{eqnarray}
\label{eq: borne convolution}
&&  \rho^{(n)}*\iota_\delta (x) - \rho^{(n)} (x) \\
&& = \int_{-1}^1 dy \,  \iota_\delta (x- y) [ \rho^{(n)} (y) -  \rho^{(n)} (x)]
= \int_{-1}^1 dy \,  \iota_\delta (x- y) \int_x^y du \,  \partial_u \rho^{(n)} (u)  \nonumber \\
&& \leqslant 
\sqrt{\int_{-1}^1 du \,  (\partial_u \rho^{(n)} (u))^2 } \;
\int_{-1}^1 dy \,  \iota_\delta (x- y) \sqrt{|x-y|} \leqslant 2 \sqrt{\gd}
\sqrt{\int_{-1}^1 du \,  (\partial_u \rho^{(n)} (u))^2 }
\, .  \nonumber
\end{eqnarray}
Using the uniform bound (\ref{bornegradrhon})  there is a constant $C$ independent of $n$ such that 
\begin{eqnarray}
\label{eq: approx id}
\int_0^t \int_{-1}^1
\big( \rho^{(n)}*\iota_\delta-\rho^{(n)} \big)^2 \leq C \gd \, .
\end{eqnarray}

It remains to check \eqref{eq: converge unif}.
Let $\frm$ be the uniform measure on $[0,t] \times [-1,1]$ given by $ \frac{1}{2t} ds  dx$.
Given $\psi (s,x)$ a non-negative function on $[0,t] \times [-1,1]$ such that 
$ \frm \big( \psi \big) = 1$, then the entropy inequality implies that for any function $\gp$
\begin{eqnarray}
\label{eq: entropy ineq}
\frm \big( \gp(s,x) \psi(s,x) \big) 
\leqslant 
\frm \big( \psi(s,x) \log \big( \psi(s,x) \big) \big)   + \log \frm \big( \exp \big(   \gp(s,x) \big) \big) \, .
\end{eqnarray}
Since $A$ and $C$ are positive then $\frm( e^{|G|} |G|)$ is bounded and \eqref{eq: entropy ineq} implies
\begin{eqnarray*}
&& 
\frm \left(  | C(\rho^{(n)} *\iota_\delta) - C(\rho^{(n)})| |\varphi| \frac{e^{|G|}}{\frm(e^{|G|}) } \right) \\
&& \qquad 
\leqslant 
\gep_\gd \,   \frm \left(  \frac{e^{|G|}}{\frm(e^{|G|}) }  \log \frac{e^{|G|}}{\frm(e^{|G|}) }  \right)
+ 
\gep_\gd  \log  \frm \left( \exp \big( \frac{1}{\gep_\gd} | C(\rho^{(n)}*\iota_\delta) - C(\rho^{(n)})| |\varphi| \big) \right) \, ,
\end{eqnarray*}
where $\gep_\gd$ vanishes (in a suitable way to be determined later) as $\gd$ goes to 0.
Recall that 
$$
\forall x \in [0,M], \quad \exp(x) \leq 1 + x + \exp(M) x^2 \qquad
{\rm and} \quad \log(x) \leq x-1 \,. 
$$
As $| C(\rho^{(n)}*\iota_\delta) - C(\rho^{(n)})| |\varphi| \leq 2 \| C\|_\infty \| \varphi \|_\infty$, the previous inequalities imply 
\begin{eqnarray*}
&& 
 \log  \frm \left( \exp \big( \frac{1}{\gep_\gd} | C(\rho^{(n)}*\iota_\delta) - C(\rho^{(n)})| |\varphi| \big) \right)\\
&& 
\leqslant 
\frac{1}{\gep_\gd^2 }  \exp \left( \frac{\| C\|_\infty \| \varphi \|_\infty }{ \gep_\gd} \right) 
\frm \left( | C(\rho^{(n)}*\iota_\delta) - C(\rho^{(n)})| |\varphi| 
+  \Big(  | C(\rho^{(n)}*\iota_\delta) - C(\rho^{(n)})| |\varphi| \Big)^2
\right) \, ,
\end{eqnarray*}
We choose  $\gep_\gd$ such that 
\begin{eqnarray*}
\lim_{\gd \to 0} \gep_\gd = 0, 
\qquad
\lim_{\gd \to 0} 
\frac{\sqrt{\gd} }{\gep_\gd}  \exp \left( \frac{\| C\|_\infty \| \varphi \|_\infty }{ \gep_\gd} \right) = 0 \, .
\end{eqnarray*}
Combining the previous inequality and (\ref{eq: approx id}), we conclude that  \eqref{eq: converge unif} holds.

We turn now to the non-linearity $\gs$. Using the decomposition
\begin{eqnarray*}
&& \int_0^t \int_{-1}^1 \, \big[  \gs \big( \rho^{(n)} \big)-\gs(R) \big] \nabla \varphi_s \, \nabla H_s 
=
\int_0^t \int_{-1}^1 \, \big[  \gs \big( \rho^{(n)} \big) - \gs \big( \rho^{(n)}*\iota_\delta \big) \big] \nabla \varphi_s \, \nabla H_s \\
&& \quad +
\int_0^t \int_{-1}^1 \, \big[  \gs \big( \rho^{(n)}*\iota_\delta \big) - \gs \big( R*\iota_\delta \big) \big] \nabla \varphi_s \, \nabla H_s 
+ \int_0^t \int_{-1}^1 \, \big[  \gs \big( R*\iota_\delta \big)-\gs(R) \big] \nabla \varphi_s \, \nabla H_s 
\end{eqnarray*}
the last two terms vanish in the limit as in \eqref{eq: decomposition}. The first term can be controlled thanks to the uniform bound \eqref{eq: approx id} and a Cauchy-Schwartz estimate as $\nabla \varphi \, \nabla H$ belongs to  $\bbL^2$.
This concludes the convergence of \eqref{edpenN}.

\medskip

Following the same proof, \eqref{eq: weak approx n}  implies that $(Q^{(n)}, K^{(n)})$ converges weakly to $(Q,K)$.

\subsubsection{$I-$convergence}

We finally complete the proof of Lemma \ref{lem: I-regularite}.
Let $(\rho^{(n)},Q^{(n)},K^{(n)})$ be the regularizing sequence defined in \eqref{edp}.
We start by proving 
\begin{eqnarray}
\limsup_{n \to \infty} I_0 (\rho^{(n)},Q^{(n)},K^{(n)}) \leqslant I_0(\rho,Q,K) \, .
\label{eq: lower bound regul}
\end{eqnarray}
From the expression (\ref{bonIun}) of $I_1$ and the weak convergence of $\rho^{(n)}$ to $\gr$, we get
\begin{eqnarray*}
I_1(\rho,Q)
&=& \int_0^T\int_{-1}^1 \frac{\sigma(\rho(s,x))}{2} \nabla H(s,x)^2  dx ds
= \lim_{\delta\to 0}  \int_0^T\int_{-1}^1 \dfrac{\sigma([\rho_s*\iota_\delta](x))}{2} \nabla H(s,x)^2 dx ds\\
&=&\lim_{\delta\to 0}\lim_{n\to\infty}\dfrac{1}{2}
\int_0^T\int_{-1}^1 \nabla H(s,x)^2 \sigma([\rho_s^{(n)}*\iota_\delta](x)) \, dx ds\\
&\geqslant& \lim_{\delta\to 0}\limsup_{n\to\infty}\dfrac{1}{2}
\int_0^T\int_{-1}^1 (\nabla H(s,x))^2 \big[ \sigma(\rho^{(n)}_s)*\iota_\delta(x) \big] \, dx ds \, ,
\end{eqnarray*}
where we used the concavity of $\sigma$ in the last inequality.
Notice that
\begin{eqnarray*}
&& \int_0^T\int_{-1}^1 (\nabla H(s,x))^2 \big[ \sigma(\rho^{(n)}_s)*\iota_\delta(x) \big] \, dx ds
= \int_0^T\int_{-1}^1 (\nabla H(s,x))^2 \sigma \big( \rho^{(n)}_s (x) \big) \, dx ds\\ 
&& \qquad \qquad \qquad 
+  \int_0^T\int_{-1}^1 \Big[ (\nabla H_s)^2*\iota_\delta(x) 
-  (\nabla H(s,x))^2 \Big] \sigma \big( \rho^{(n)}_s (x) \big) \, dx ds \, .
\end{eqnarray*}
As $\gs$ is bounded and $(\nabla
H)^2$ is integrable, the last term vanishes uniformly in $n$ when $\gd$ goes to 0, so that 
\begin{eqnarray*}
I_1(\rho,Q) \geqslant  \limsup_{n\to\infty}\dfrac{1}{2}\int_0^T\int_{-1}^1 \nabla
H(s,x)^2\sigma \big( \rho^{(n)}_s(x) \big) dx\ ds.
\end{eqnarray*}
As $\nabla H^{(n)}$ converges to $\nabla H$ in $\bbL^2$ and $\sigma$ is bounded, we conclude that
\begin{equation}
I_1(\rho,Q) \geq \limsup_{n \to \infty} I_1(\rho^{(n)},Q^{(n)}) \, .
\label{Iungrand}
\end{equation}

Similarly for $I_2$, the concavity and  boundedness of $C$ and $A$
imply that 
$$
 I_2(\rho,K)\geqslant \limsup_{n \to \infty}  \int_0^T\int_{-1}^1C(\rho^{(n)})\
 \left(1-e^{G}+G e^{G}\right)
+A(\rho^{(n)})\ \left(1-e^{-G}-G e^{-G}\right).$$
Then we use
the convergence properties on $G^{(n)}$ (\ref{Gconvergence}) and the
fact that $C$ and $A$ are bounded to conclude that
\begin{equation}
I_2(\rho,K)\geqslant \limsup_{n \to \infty} 
I_2(\rho^{(n)},K^{(n)})\label{Ideuxgrand}
\end{equation}
Estimates (\ref{Iungrand}) and (\ref{Ideuxgrand}) imply 
\eqref{eq: lower bound regul}.

The converse inequality
\begin{eqnarray}
\liminf_{n \to \infty} I_0 (\rho^{(n)},Q^{(n)},K^{(n)}) \geq I_0(\rho,Q,K) \, ,
\label{eq: lower bound regul 2}
\end{eqnarray}
can be deduced from the weak convergence of $(\rho^{(n)},Q^{(n)},K^{(n)})$ to $(\rho,Q,K)$ and from
the decomposition \eqref{tildei}: the limit follows for the term $\tilde{I}_0$ from its lower semi-continuity
and the convergence of the second term $\int_0^T \, dt \, \bra C(\rho^{(n)}_t ) + A(\rho^{(n)}_t) \ket $ can be obtained 
as in \eqref{eq: decomposition 0}.

\section{The density large deviations}
\label{sec: densiteGD}

In this section, we recover the density large deviation principle 
(first derived in \cite{JLV}) by optimizing the functional $I$ over the currents. 
We assume that the rates $A,C$ satisfy assumptions ({\bf L1, L2}).
The contraction principle \cite{DZ}  implies that the density large deviation functional is given by 
\begin{equation}
\label{eq: LD density}
\cF (\gr) = \inf_{ (Q,K) } I(\gr,Q,K) \, ,
\end{equation}
where the infimum is taken over the currents $(Q,K)$.
Using the approximation procedure of section \ref{sec: Approximation for general trajectories}, we will check that 
it is enough to consider regular density profiles and the modified functional 
\begin{equation}
\label{eq: LD density 1}
\hat \cF (\gr) = \inf_{ (Q,K)_\reg } I(\gr,Q,K) \, ,
\end{equation}
where the infimum is taken now over the regular currents $(Q,K)$.

\medskip

\noindent
{\bf Step 1.}
In this first step, the explicit solution of the variational problem \eqref{eq: LD density 1} is computed for regular trajectories.
The functional \eqref{eq: LD density 1} can be rewritten as (see \eqref{rem: regularsolution})
\begin{eqnarray}
\label{eq: GD temp} 
\hat \cF  (\gr) = \inf_{ (G,H) }
\left \{
\dfrac{1}{2}\int_0^T
\left< \gs( \rho_s ) \,  |\nabla H_s|^2 \right>ds +
\int_0^T\left<C(\rho_s) \, \left(1-e^{G_s}+G_se^{G_s}\right)\right> \, ds \right.
\nonumber \\
\left.
+\int_0^T\left<A(\rho_s)\
\left(1-e^{-G_s}-G_se^{-G_s}\right)\right> \, ds \right \} \, ,
\end{eqnarray} 
where $G,H$ are smooth functions such that the conservation relation \eqref{relation} holds
\begin{equation}
\label{eq: constraint H G}
\partial_s  \rho_s  =
\dfrac{1}{2} \Delta \rho_s - \partial_x \Big( \gs \big( \rho_s \big) \nabla H_s \Big)  + \big( C(\rho_s)e^{G_s}-A(\rho_s)e^{-G_s} \big) .
\end{equation}

\medskip

We first check that the infimum is reached for functions $H$ with boundary conditions $H(s,-1) = H(s,1) =0$ for any time $s$.
Let $f$ be a smooth function in $[0,T]$.
Perturbing $H$ into $H(s,x) + f(s) \int_0^x \frac{1}{\gs(\gr(s,u))} \, du$, we see that the conservation law \eqref{eq: constraint H G} is preserved (it simply amounts to adding a constant conservative current) and 
\begin{eqnarray*}
&& \int_0^T \left< \gs( \rho_s ) \,  \left| \nabla H_s + \frac{\ga \, f (s) }{\gs(\gr(s,x))} \right|^2 \right>ds 
= \int_0^T \left< \gs( \rho_s ) \,  \left| \nabla H_s \right|^2 \right>ds 
\\
&& \qquad \qquad + 2 
\int_0^T \big( H(s,1) - H(s,-1) \big) f(s)
+  \int_0^T \left \bra  \frac{\big( f(s) \big)^2}{\gs(\gr(s,x))} 
\right \ket \, ds \, .
\end{eqnarray*} 
As $H$ minimizes the integral, this implies that $H$ vanishes at the boundaries.

\medskip

Suppose that an extremum is reached at $(H,G)$ and consider a perturbation with the new drifts $H +h$ and $G+ g$.
Then the constraint \eqref{eq: constraint H G} implies that $h,g$ satisfy the relation
\begin{equation}
\label{eq: constraint peturb}
\partial_x \Big(  \gs \big( \rho_s \big) \nabla h_s \Big) 
+  g_s \big ( C(\rho_s)e^{G_s}+A(\rho_s)e^{-G_s} \big) = 0 .
\end{equation}
A perturbation of \eqref{eq: GD temp} around the extremum  $H,G$ leads to 
\begin{eqnarray}
\label{eq: 106}
0= \int_0^T
\left< \gs(\rho_s) \, \nabla H_s \, \nabla h_s \right> \, ds +
\int_0^T\left< g_s G_s \left(C(\rho_s) e^{G_s} +A(\rho_s) e^{-G_s}\right) \right> \, ds \, .
\end{eqnarray} 
Since $H$ vanishes at the boundaries, 
the relation \eqref{eq: constraint peturb}  combined to \eqref{eq: 106} leads to 
\begin{eqnarray*}
0= 
\int_0^T\left< g_s [G_s - H_s] \left(C(\rho_s) e^{G_s} +A(\rho_s) e^{-G_s}\right) \right> \, ds \, .
\end{eqnarray*} 
This holds for any $g$ so that  the extremum  is such that $G=H$ with $H$ determined by 
\begin{equation}
\label{eq: minimum G,H}
\partial_t \rho=\dfrac{1}{2}\Delta \rho
-\nabla\left(\sigma(\rho)\nabla H\right)+C(\rho)e^H-A(\rho)e^{-H} \, .
\end{equation}
Thus if $H$ satisfies \eqref{eq: minimum G,H} and $G = H$, then an extremum is reached for the corresponding currents $(\hat Q, \hat K)$.
Since the functional $I(\gr,Q,K)$ is convex wrt $(Q,K)$ (thanks to the representation
\eqref{J1}, \eqref{J2}), the extremum $(\hat Q, \hat K)$ has to be a global minimum, thus
\begin{eqnarray*}
\hat \cF (\gr) = 
\dfrac{1}{2}\int_0^T
\left< \gs( \rho_s ) \,  |\nabla H_s|^2 \right>ds +
\int_0^T\left<C(\rho_s) \, \left(1-e^{H_s}+ H_s e^{H_s} \right)\right> \, ds
\nonumber \\
+\int_0^T\left<A(\rho_s)\
\left(1-e^{-H_s}- H_s e^{-H_s}\right)\right> \, ds \, .
\end{eqnarray*}

\medskip

\noindent
{\bf Step 2.}
To approximate $\cF$ \eqref{eq: LD density} in terms of $\hat \cF$ \eqref{eq: LD density 1}, we first check that the minimum is reached in \eqref{eq: LD density}.
Consider a sequence $(Q^n,K^n)$ which realizes the infimum. By the tightness argument (section \ref{fermes})
the sequence belongs to a compact set and therefore has a weak limit $(Q^\star,K^\star)$. 
From the lower semi-continuity of the functional $(Q,K) \to I(\gr,Q,K)$,  this weak limit is a minimizer
\begin{equation}
\label{eq: LD density minimum}
\cF (\gr) =  I(\gr,Q^\star,K^\star) \, .
\end{equation}

Let $(\gr_\gep, Q^\star_\gep,K^\star_\gep)$ be a regular sequence (as in \eqref{edp}) converging to  $(\gr, Q^\star,K^\star)$ such that 
\begin{equation*}
I (\gr_\gep, Q^\star_\gep,K^\star_\gep) \to  I(\gr,Q^\star,K^\star) \, .
\end{equation*}
As $\hat \cF (\gr_\gep) \leq  I (\gr_\gep, Q^\star_\gep,K^\star_\gep)$, one has
\begin{equation*}
\limsup_{\gep \to 0} \hat \cF (\gr_\gep)
\leq 
\lim_{\gep \to 0} I (\gr_\gep, Q^\star_\gep,K^\star_\gep) =  \cF (\gr)  \, .
\end{equation*}

Using \eqref{eq: minimum G,H}, there are regular $(\hat Q_\gep, \hat K_\gep)$
such that 
\begin{equation*}
\hat \cF (\gr_\gep) =  I (\gr_\gep, \hat Q_\gep, \hat K_\gep) \, .
\end{equation*}
The sequence $(\gr_\gep, \hat Q_\gep, \hat K_\gep)$ has a bounded large deviation cost and thus
it belongs to a compact set.
There is a subsequence such that $(\hat Q_\gep, \hat K_\gep)$ converges weakly to $(\hat Q, \hat K)$.
One gets
\begin{equation*}
\liminf_{\gep \to 0} \hat \cF (\gr_\gep) \geq   I (\gr, \hat Q, \hat K) \geq \cF (\gr) \, .
\end{equation*}
This limit follows from the decomposition \eqref{tildei}. 
The term $\tilde{I}_0$ converges by lower semi-continuity and the convergence of the second term $\int_0^T \, dt \, \bra C(\gr_\gep (t) ) + A( \gr_\gep (t)) \ket $ can be obtained 
as in \eqref{eq: decomposition 0} ($\gr_\gep$ is just a subsequence extracted from $\gr^{(n)}$).

Combining both estimates, we deduce that $\cF$ can be approximated by $\hat \cF$.

\section{Appendix : Uniqueness of the weak solutions}
\label{app: Uniqueness}

Let $H$ be in $\bbL^2([0,T],\bbH_1(]-1,1[) )$ and $|G| \exp(|G|)$ in $\bbL^1([0,T] \times ]-1,1[)$.
Given an initial data $\gga$, a weak solution of 
\begin{equation}
\left\{\begin{array}{lcl}
\partial_t  \rho &=& \frac{1}{2}\Delta  \rho 
- \partial_x \left( \gs \big( \rho \big) \partial_x H \right)
+ C(\rho)e^{G}-A(\rho)e^{-G}\\
\bar  \rho(t,\pm 1)&=& \bar  \rho_\pm, \qquad \bar  \rho(0,x) = \gamma(x)
\end{array}\right.  
\label{limitehydro general}
\end{equation}
is defined as :
\begin{itemize}
\item 
The density $\gr$ is in $\bbL^2 \big( [0,T], \bbH^1(]-1,1[) \big)$, i.e. there is 
a function in $\mathbb{L}^2 ([0,T] \times ]-1,1[)$ which will be denoted by 
$\nabla \rho$ such for every $t\in[0,T]$ and every function $\varphi\in C^1([-1,1])$
\begin{equation}
\label{eq: weak bd 2}
\int_0^t \, ds \, \int_{-1}^1 dx\rho(s,x)\nabla
\varphi(x) -\{\bar \gr_+\varphi(1)-\bar \gr_-\varphi(-1)\}t 
= \int_0^t \, ds \, \int_{-1}^1dx \, \nabla \rho (s,x) \varphi(x) \, ,
\end{equation}
where $\bar \gr_\pm$ are fixed boundary conditions.
\item
For every $t\in[0,T]$ and every function $\varphi\in C^1([-1,1])$
vanishing at the boundaries,
\begin{multline}
\int_{-1}^1dx\rho(t,x)\varphi(x)-\int_{-1}^1dx\gamma(x)\varphi(x)
= - \int_0^tds\int_{-1}^1 \, dx \, \nabla \rho (s,x)\nabla \varphi(x)\\
+\int_0^tds\int_{-1}^1 dx \, \gs \big( \rho(s,x) \big)\nabla H(s,x)\nabla \varphi(x)
\\
+\int_0^tds\int_{-1}^1dx\left(C(\rho(s,x))e^{G(s,x)}-A(\rho(s,x))e^{-G(s,x)}\right)\varphi(x) \, .
\label{weaksoldrift 2}
\end{multline}
\end{itemize}
The hydrodynamic limit \eqref{limitehydro} corresponds to $H = G = 0$.

\medskip

In this Appendix, we derive the uniqueness of the weak solutions.
The main technical difficulty  comes from the fact that $G, \partial_x H$ are unbounded
(see  \cite{evans} for bounded drifts).  Note that at this stage Assumption $(\textbf{L1})$ is irrelevant.
We will rely on Assumption $(\textbf{L2})$ on $A,C$ which can be interpreted as follows. 
Equation \eqref{limitehydro general} in a strong form reads 
\begin{eqnarray*}
\partial_t \rho(t,x)= \Delta \rho (t,x) - \partial_x \big(   \gs \big( \rho(t,x) \big) \nabla H(t,x) \big)
- V^\prime_{G(t,x)} (\rho(t,x)) \, ,
\end{eqnarray*}
where the reaction term is determined by the space-time dependent potential $V_{G(t,x)}$ with  $V^\prime_{g} (\rho) = - e^g C(\rho) + e^{-g} A(\rho)$. Assumption {\bf (L2)} ensures that the potential $V_{G(t,x)}$ is convex uniformly in $G(t,x)$.
Thus the reaction and the diffusion terms are both contractions and  the solution will be unique.
 We adapt to our framework the argument of \cite{LMS, FLM}.

\medskip

We consider two initial datas $\gr^1_0,\gr^2_0$ and the corresponding weak solutions $\gr^1 ,\gr^2$. We are going to prove that the $\bbL^1$-norm $\| \gr^1_t  - \gr^2_t\|_1$ decreases in time.
For a given $\gd>0$, we introduce the regularized absolute value
\begin{eqnarray}
\label{eq: cut-off}
U_\gd (v) = \frac{v^2}{2 \gd} 1_{ \{ |v| \leq \gd \} } + \Big(|v|-\frac{\gd}{2} \Big) 
1_{ \{ |v| > \gd \} } \, .
\end{eqnarray}
Define
\begin{eqnarray}
\label{eq: subset}
V_\gd  = \Big \{ (s,x) \in [0,T] \times [-1,1], \qquad \text{such that} \quad
| \gr^1(s,x)  - \gr^2(s,x)| \leq \gd \, \Big \}  \, .
\end{eqnarray}

\medskip

\noindent
{\bf Step 1.} 
We are going to check that for times $t<t'$
\begin{multline}
\label{eq: integration parties}
\int_{-1}^1 \, dx\, U_\gd \Big( \gr^1(t',x)  - \gr^2 (t',x) \Big) 
- \int_{-1}^1 \, dx\, U_\gd \Big( \gr^1 (t,x)  - \gr^2 (t,x) \Big) \\
 = - \frac{1}{\gd} 
\int_t^{t'}  \int_{-1}^1 ds dx \, 
1_{ \{ (x,s) \in V_\gd \} } \nabla ( \gr^1- \gr^2) (s,x) \,
\Big\{  \nabla ( \gr^1 - \gr^2)(s,x)  + \bar \gs(s,x)  \nabla H(s,x) \Big\}\\
+ \frac{1}{\gd} \int_t^{t'}  \int_{-1}^1  ds dx \, U_\gd^\prime \big( \gr^1(s,x)  - \gr^2(s,x) \big)
\Big\{ \bar C(s,x) e^{G(s,x)} - \bar A(s,x) e^{-G(s,x)}  \Big\} 
\end{multline}
where we set 
\begin{eqnarray*}
\bar A (t,x) &=& A(\gr^1(t,x))- A(\gr^2(t,x)), \quad
\bar C (t,x) = C(\gr^1(t,x)) - C(\gr^2(t,x)),  \\
\bar \gs (t,x) &=& \gs(\gr^1(t,x)) - \gs(\gr^2(t,x)) \, .
\end{eqnarray*}

To prove \eqref{eq: integration parties}, we follow the regularization scheme introduced in \cite{BLM} (see the proof of their Theorem 4.6).
For $\gep>0$, denote by $R^D_\gep$ : $[-1,1]^2 \to \bbR^+$ (resp $R^N_\gep$) the resolvent of the Dirichlet  Laplacian $\gD_D$ (resp Neumann $\gD_N$)
\begin{eqnarray*}
R^D_\gep = ({\rm Id} - \gep \gD_D)^{-1}, \qquad
R^N_\gep = ({\rm Id} - \gep \gD_N)^{-1} \, .
\end{eqnarray*}
The mollified trajectory is defined by 
\begin{eqnarray*}
\gr^\gep(t,x) = \bar \gr(x) +  R^D_\gep \big( \gr_t - \bar \gr)(x) \, ,
\end{eqnarray*}
where $\bar \gr$ stands for the linear profile between $\bar \gr_+$ and $\bar \gr_-$. 
Note that the resolvent of the Dirichlet  Laplacian preserves the boundary conditions of the mollified trajectory. 

As $\gr$ is a weak solution of \eqref{weaksoldrift 2}, one has
\begin{eqnarray*}
\partial_t \gr^\gep(t,x) &=&  \partial_t \int_{-1}^1 dy \, R^D_\gep (x,y) \big( \gr(t,y) - \bar \gr(y)) \\
&=&  \int_{-1}^1 dy \, \partial_y R^D_\gep (x,y) \Big\{ - \partial_y \big( \gr(t,y) - \bar \gr(y))
+ \gs\big( \rho(t,y)\big) \nabla H(t,y) \Big\}\\
&& \qquad \qquad + \int_{-1}^1 dy \, R^D_\gep (x,y) \left(C(\rho(t,y))e^{G(t,y)}-A(\rho(t,y))e^{-G(t,y)}\right) \, ,
\end{eqnarray*}
where we used the fact that $R^D_\gep (x,1) = R^D_\gep (x,-1) = 0$ for any $x$.

From the relation $\partial_y R^D_\gep (x,y) = - \partial_x R^N_\gep (x,y)$, one has
\begin{eqnarray}
\partial_t \gr^\gep(t,x) &=&  \partial_x \int_{-1}^1 dy \,  
R^N_\gep (x,y) \Big\{ \partial_y \big( \gr(t,y) - \bar \gr(y))
- \gs\big( \rho(t,y)\big) \nabla H(t,y) \Big\} \nonumber\\
&& \qquad \qquad + \int_{-1}^1 dy \, R^D_\gep (x,y) \left(C(\rho(t,y))e^{G(t,y)}-A(\rho(t,y))e^{-G(t,y)}\right) \, .
\label{eq: mollif}
\end{eqnarray}

Taking the time derivative we obtain
\begin{eqnarray*}
\partial_t
\int_{-1}^1 \, dx\, U_\gd \Big( \gr^{1,\gep}(t,x)  - \gr^{2,\gep}(t,x) \Big) 
= 
\int_{-1}^1 \, dx\, U_\gd^\prime \Big( \gr^{1,\gep}(t,x)  - \gr^{2,\gep}(t,x) \Big) 
\partial_t \Big( \gr^{1,\gep}(t,x)  - \gr^{2,\gep}(t,x) \Big) \\
\end{eqnarray*}
As $\partial_x R^D_\gep (x,y) = - \partial_y R^N_\gep (x,y)$, we get 
\begin{eqnarray*}
\partial_x  \, \big( \gr^{1,\gep}(t,x)  - \gr^{2,\gep}(t,x) \big)
&=& - \int_{-1}^1 dy \, \partial_y R^N_\gep(x,y) ( \gr^1 (t,y) - \gr^2  (t,y) ) \\
&=&  \int_{-1}^1 dy \,  R^N_\gep(x,y) (\partial_y \gr^1 (t,y) - \partial_y \gr^2  (t,y) )  \, .
\end{eqnarray*}
where we used that $\gr^1,\gr^2$ are in $\bbL^1 \big( [0,T],\bbH_1(]-1,1[) \big)$.
Thus we can write
\begin{eqnarray*}
\partial_x  \, \Big[ U_\gd ' \big( \gr^{1,\gep}(t,x)  - \gr^{2,\gep}(t,x) \big) \Big]
=  1_{\{x \in V_{\gd,t}^\gep \}} R^N_\gep \big( \partial_x ( \gr^1  - \gr^2 \big) \big) (t,x) \, ,
\end{eqnarray*}
with the notation 
\begin{eqnarray*}
\label{eq: subset}
V_{\gd,t}^\gep  = \Big \{ x \in [-1,1], \qquad \text{such that} \quad
| \gr^{1,\gep}(t,x)  - \gr^{2,\gep}(t,x)| \leq \gd \, \Big \}  \, .
\end{eqnarray*}

Using relation \eqref{eq: mollif}, one obtains
\begin{eqnarray*}
&& \partial_t
\int_{-1}^1 \, dx\, U_\gd \Big( \gr^{1,\gep}(t,x)  - \gr^{2,\gep}(t,x) \Big) \\
&& \qquad = - \int_{-1}^1 \, dx\, \int_{-1}^1 dy \,1_{\{x \in V_{\gd,t}^\gep \}} R^N_\gep \big( \partial_x ( \gr^1  - \gr^2 \big) \big) (t,x)
\\
&& \qquad \qquad \qquad 
  R^N_\gep (x,y) \Big\{ \partial_y \big( \gr^1 (t,y)  - \gr^2 (t,y) \big) 
+  \bar \gs (t,y)  \nabla H(t,y) \Big\} \\
&& \qquad  + \iint_{-1}^1 \, dx dy \,
U_\gd^\prime \big( ( \gr^{1,\gep}  - \gr^{2,\gep})(t,x) \big) \,
R^D_\gep (x,y) \left(\bar C (t,y) e^{G(t,y)}- \bar A (t,y) e^{-G(t,y)}\right) \, .
\end{eqnarray*}
Taking the limit as $\gep$ tends to 0, we recover \eqref{eq: integration parties}.

\medskip

\noindent
{\bf Step 2.} 

First note that from assumption {\bf (L2)}, one has
\begin{eqnarray*}
U_\gd^\prime \big( ( \gr^1  - \gr^2)(t,x) \big) \left(\bar C (t,y) e^{G(t,y)}- \bar A (t,y) e^{-G(t,y)}\right) \leq 0 \, .
\end{eqnarray*}
Thus the reaction term acts as a contraction.
As  $\gs(\gr) = \gr(1-\gr)$, then 
\begin{eqnarray*}
\forall (s,x) \in V_\gd, \qquad
\big| \bar \gs (s,x) \big| \leq \gd \, .
\end{eqnarray*}
Thus \eqref{eq: integration parties} implies
\begin{eqnarray}
\label{eq: borne d'energie}
&& \bra U_\gd \big( \gr^1_{t'}  - \gr^2_{t'} \big) \ket
- \bra U_\gd \big( \gr^1_t  - \gr^2_t \big) \ket 
\leq - \frac{1}{\gd} 
\int_t^{t'} \, ds \big \bra 1_{ \{ (x,s) \in V_\gd \} } \big( \nabla   \gr^1_s - \gr^2_s \big)^2 \big \ket \\
&& \qquad \qquad +    \int_t^{t'} ds \big \bra 1_{ \{ (x,s) \in V_\gd \} }
| \nabla  (\gr^1_s - \gr^2_s)| \,   |\nabla H_s| \ket
\, , \nonumber
\end{eqnarray}
using the fact that 
\begin{eqnarray*}
2 \big| \nabla  (\gr^1 - \gr^2) (s,x) \big| \,   |\nabla H(s,x)| 
\leq  
\frac{1}{\gd} \big| \nabla  (\gr^1 - \gr^2) (s,x) \big|^2  +  \gd  |\nabla H(s,x)|^2  \, ,
\end{eqnarray*}
we obtain
\begin{eqnarray*}
\bra U_\gd \big( \gr^1_{t'}  - \gr^2_{t'} \big) \ket
- \bra U_\gd \big( \gr^1_t  - \gr^2_t \big) \ket 
&\leq& - \frac{1}{ 2 \gd} 
\int_t^{t'} ds \, \bra  1_{ \{ (x,s) \in V_\gd \} } \; \big( \nabla   \gr^1_s - \gr^2_s \big)^2 \big \ket 
+   \gd  \int_t^{t'}  ds \, \bra  |\nabla H_s |^2 \ket \\
&& \leq 
  \gd \int_t^{t'}  ds \, \bra  |\nabla H_s |^2  \ket \, .
\end{eqnarray*}
Recall that $H$ belongs to $\bbL^2([0,T],\bbH_1(]-1,1[)$ thus as $\gd$ tends to 0, the LHS converges to 0. Furthermore $U_\gd$ converges to the absolute value function. This implies that
\begin{eqnarray*}
\forall t \leq t', \qquad 
\int_{-1}^1 \, dx\, \Big| \gr^1 (t',x)   - \gr^2 (t',x)   \Big| 
\leq  \int_{-1}^1 \, dx\, \Big| \gr^1 (t,x)  - \gr^2 (t,x) \Big| \, ,
\end{eqnarray*}
from which the uniqueness of the weak solutions follows. \qed

\end{document}